\documentclass[12pt,psamsfonts]{amsart}
\usepackage[left=3cm,top=2.5cm,right=3cm,bottom=2.5cm]{geometry}

\usepackage{amsmath,amssymb}
\usepackage{graphicx}
\usepackage{color}
\usepackage{psfrag} 
\usepackage{overpic} 
\usepackage{doi}
 
\newtheorem{theorem}{Theorem}
\newtheorem{proposition}[theorem]{Proposition}
\newtheorem{corollary}[theorem]{Corollary}
\newtheorem{lemma}[theorem]{Lemma}
\newtheorem {remark}[theorem]{Remark}

\title[]{On Cyclicity in Discontinuous Piecewise Linear near-Hamiltonian Differential Systems with three Zones having a Saddle in the Central one}
\author[R. Euz\'ebio, M. Gouveia, D. Novaes, C. Pessoa and R. Ribeiro]{}

  \subjclass[2021]{34C07}
   \keywords{Limit Cycles; Piecewise Hamiltonian differential system; Melnikov function; Periodic Annulus}

\begin{document}
 \maketitle

\centerline{Rodrigo Euz\'ebio$^1$, M\'arcio Gouveia$^2$, Douglas Novaes$^3$, Claudio Pessoa$^4$ and Ronisio Ribeiro$^5$}

\medskip

{\footnotesize \centerline{Universidade Federal de Goiâs (UFG),} \centerline{Instituto de Matemática e Estatística,} \centerline{R. Jacarandá, 74.690-900, Goiânia, GO, Brazil }
	\centerline{\email{euzebio@ufg.br$^1$}}}

\bigskip

{\footnotesize \centerline{Universidade Estadual Paulista (UNESP),} \centerline{Instituto de Bioci\^encias Letras e Ci\^encias Exatas,} \centerline{R. Cristov\~ao Colombo, 2265, 15.054-000, S. J. Rio Preto, SP, Brazil }
\centerline{\email{mra.gouveia@unesp.br$^2$}, \email{c.pessoa@unesp.br$^4$} and \email{ronisio.ribeiro@unesp.br$^5$}}}

\bigskip

{\footnotesize \centerline{Universidade Estadual de Campinas (UNICAMP),} \centerline{Departamento de Matemática,} \centerline{R. Sérgio Buarque de Holanda, 651, 13.083-859, Campinas, SP, Brazil }
\centerline{\email{ddnovaes@ime.unicamp.br$^3$}}}

\medskip

\bigskip

\begin{quote}{\normalfont\fontsize{8}{10}\selectfont
{\bfseries Abstract.} In this paper, we study the number of limit cycles that can bifurcate from a periodic annulus of discontinuous planar piecewise linear Hamiltonian differential system with three zones separated by two parallel straight lines, such that the linear differential system, given by the piecewise one, in the region between the two straight lines (called of central subsystem) has a saddle at a point equidistant from these lines (obviously, the others subsystems have saddles and centers). We prove that the maximum number of limit cycles that bifurcate from the periodic annulus of this kind of piecewise Hamiltonian differential systems, by linear perturbations, is at least six. For this, we obtain normal forms for the systems and study the number of zeros of its Melnikov functions defined in two and three zones.


\par}
\end{quote}

\section{Introduction and Main Result}
A lot of theory has been developed to investigate the number and position of limit cycles from polynomial differential systems. Having as main motivation the search for the solutions of the famous Hilbert’s 16th problem, which was propose in 1900, see \cite{Hil02}. Currently the study of the limit cycles has been considered for piecewise differential systems in regions separated by a straight line. These systems appear in a natural way in mechanics, electrical circuits, control theory, neurobiology, etc (see the books \cite{diB08, Hen97, Nar14} and the papers \cite{Chu90, Fit61, McK70, Nag62}).
Assuming that the subsystems that define the piecewise one are linear, in the continuous case, was proved in \cite{Fre98} that such systems have at most one limit cycle. On the other hand, when the piecewise linear differential system is discontinuous, i.e. the subsystems do not coincide on the separation straight line, is known that the maximum number of limit cycles is at least three. Apparently, for this case, determining the exact number of limit cycles is a hard task, but important partial results have been obtained, see for instance \cite{Bra13, Buz13, Fre12, Fre14b, Li14, Lli12, LNT, Wan19}. 

Recently, piecewise differential system defined in regions with more than two zones have attracted the attention of researchers, see \cite{Don17, Fon20, Hu13, Li21, Lli15b, Lli18a, Wan16, Yan20}. Results imposing restrictive hypotheses on the systems, such as symmetry and linearity, have been obtained. For instance, conditions for nonexistence and existence of one, two or three limit cycles for symmetric continuous piecewise linear differential systems with three zones can be found in \cite{Lli14}. Now, for the nonsymmetric case, examples with two limit cycles surrounding a unique singular point at the origin was found in \cite{Lli15, Lim17}. Removing these restrictions, we have recent papers estimating the lower bounds for the number of limit cycles bifurcating, by linear perturbations, from periodic annulus of discontinuous piecewise linear differential system with three zones separated by two parallel straight lines. For instance, in \cite{Xio21} the authors showed that at least seven limit cycles can bifurcate from a periodic annulus when the subsystems that define the piecewise one have a boundary global center. On the other hand, when the central subsystem, i.e. the system defined between the two parallel lines, has a center at the origin and the others subsystems have centers or saddles then the maximum number of limit cycles is at least three, see \cite{Pes22a}. Now, when the central subsystem has a saddle at the origin and the others subsystems have virtual centers, there are works with five limit cycles, see \cite{Zan22}. 

 In this paper, we contribute along these lines. Our goal is estimated the lower bounds for the number of crossing limit cycles of a discontinuous piecewise linear near-Hamiltonian differential systems with three zones, given by

\begin{equation}\label{eq:01}
	\left\{\begin{array}{ll}
		\dot{x}= H_y(x,y)+\epsilon f(x,y), \\
		\dot{y}= -H_x(x,y)+\epsilon g(x,y),
	\end{array}
	\right.
\end{equation}
with
\begin{equation*}
	H(x,y)=\left\{\begin{array}{ll}\vspace{0.2cm}
		H^{\scriptscriptstyle L}(x,y)=\dfrac{b_{\scriptscriptstyle L}}{2}y^2-\dfrac{c_{\scriptscriptstyle L}}{2}x^2+a_{\scriptscriptstyle L}xy+\alpha_{\scriptscriptstyle L}y-\beta_{\scriptscriptstyle L}x, \quad x\leq -1, \\ \vspace{0.2cm}
		H^{\scriptscriptstyle C}(x,y)=\dfrac{b_{\scriptscriptstyle C}}{2}y^2-\dfrac{c_{\scriptscriptstyle C}}{2}x^2+a_{\scriptscriptstyle C}xy+\alpha_{\scriptscriptstyle C}y-\beta_{\scriptscriptstyle C}x, \quad -1\leq x\leq 1, \\
		H^{\scriptscriptstyle R}(x,y)=\dfrac{b_{\scriptscriptstyle R}}{2}y^2-\dfrac{c_{\scriptscriptstyle R}}{2}x^2+a_{\scriptscriptstyle R}xy+\alpha_{\scriptscriptstyle R}y-\beta_{\scriptscriptstyle R}x, \quad x \geq 1, \\
	\end{array}
	\right.
\end{equation*}
\begin{equation}\label{eq:02}
	f(x,y)=\left\{\begin{array}{ll}
		f_{\scriptscriptstyle L}(x,y)=r_{10}x+r_{01}y+r_{00}, \quad x\leq -1, \\
		f_{\scriptscriptstyle C}(x,y)=u_{10}x+u_{01}y+u_{00}, \quad -1\leq x\leq 1, \\
		f_{\scriptscriptstyle R}(x,y)=p_{10}x+p_{01}y+p_{00}, \quad x \geq 1, \\
	\end{array}
	\right.
\end{equation}
\begin{equation}\label{eq:03}
	g(x,y)=\left\{\begin{array}{ll}
		g_{\scriptscriptstyle L}(x,y)=s_{10}x+s_{01}y+s_{00}, \quad x\leq -1, \\
		g_{\scriptscriptstyle C}(x,y)=v_{10}x+v_{01}y+v_{00}, \quad -1\leq x\leq 1, \\
		g_{\scriptscriptstyle R}(x,y)=q_{10}x+q_{01}y+q_{00}, \quad x \geq 1, \\
	\end{array}
	\right.
\end{equation}
where the dot denotes the derivative with respect to the independent variable $t$, here called the time, and $0\leq\epsilon<<1$.  When $\epsilon=0$ we say that system \eqref{eq:01} is a piecewise Hamiltonian differential system. We call system \eqref{eq:01} of {\it left subsystem} when $x\leq -1$, {\it right subsystem} when $x\geq 1$ and {\it central subsystem} when $-1\leq x\leq 1$. Moreover, we can classify a singular point $p$ of a subsystem from \eqref{eq:01} according to its position in the regions determinate by the parallel straight lines. More precisely, the central subsystem from \eqref{eq:01} has a {\it real singular point} $p=(p_x,p_y)$ when $-1<p_x<1$, a {\it virtual singular point} when $p_x>1$ or $p_x<-1$ and a {\it boundary singular point} when $p_x=\pm1$. Accordingly, the left (resp. right) subsystem from \eqref{eq:01} has a real singular point $p=(p_x,p_y)$ when $p_x<-1$ (resp. when $p_x>1$ ), a virtual singular point when $p_x>-1$ (resp. when $p_x<1$) and a boundary singular point when $p_x=-1$ (resp. when $p_x=1$). The two parallel straight lines that separate the plane into three zones are called {\it switching lines}. Note that to assume these switching lines parallels to the $y$-axis and passing through the points of abscissa  $x=\pm 1$ is not a constraint. 
 
\medskip

In addition, let us assume that system $\eqref{eq:01}|_{\epsilon=0}$ satisfies the following hypotheses:
\begin{itemize}
	\item[{\rm (H1)}] The unperturbed central subsystem from $\eqref{eq:01}|_{\epsilon=0}$ has a real saddle and the others unperturbed subsystems from $\eqref{eq:01}|_{\epsilon=0}$ have centers or saddles.
	\item[{\rm (H2)}] The unperturbed system  $\eqref{eq:01}|_{\epsilon=0}$ has only crossing points on the straights lines $x=\pm 1$, except by some tangent points.
	\item[{\rm (H3)}] The unperturbed system $\eqref{eq:01}|_{\epsilon=0}$ has three  periodic annulus, one consisting of a family of crossing periodic orbits passing through three zones and the others two consisting of families of crossing periodic orbits passing through two zones, such that each orbit of these families has clockwise orientation.  	
\end{itemize}

By hypothesis (H1), we can classify the system $\eqref{eq:01}|_{\epsilon=0}$ according to the configuration of their singular points (without taking into account if the singular points are real or virtual). More precisely, denoting the centers by the capital letter C and by S the saddles, in the case of three zones, we have the following three class of piecewise linear Hamiltonian systems: SSS, CSS and CSC. That is, SSS indicates that the singular points of the linear systems that define the piecewise differential system are saddles and so on.


\medskip  

We will use the vector $(l, k, m)\in\mathbb{Z}^3$ to denote the limit cycle configurations of system $\eqref{eq:01}$, where its components $k$, $l$ and $m$ indicate the number of limit cycles passing through three zones, the number of limit cycles passing through two zones and surrounding the tangent point of the central subsystem with the switching line $x=1$ and the number of limit cycles surrounding the tangent point of the central subsystem with the switching line $x=-1$, respectively.

\medskip

Thus, the main result of this paper is the follow.

\begin{theorem}\label{the:01}
	The maximum number of limit cycles of system \eqref{eq:01}, satisfying hypotheses {\rm (Hi)}, for $i=1,2,3$, which can bifurcate simultaneous from the periodic annulus of the unperturbed system $\eqref{eq:01}|_{\epsilon=0}$ is at least five, in the case SSS, and at least six, in the cases CSS and CSC. Moreover, the following configurations of limit cycles are achievable: $(2, 2, 1)$ or $(2, 1, 2)$ for the case SSS and $(2,2,2)$ for the cases CSS and CSC.
\end{theorem}


\medskip

The Theorem \ref{the:01} gives us a larger lower bound for the number of limit cycles than the one obtained in the paper \cite{Zan22}. We would like to point out that in \cite{Zan22} the authors perturb a particular system of the type CSC which is symmetrical with respect to the origin. Here we work with perturbation of general families of pieciwise linear systems determined by the hypotheses (Hi), for $i=1,2,3$, without imposing restrictive hypotheses  like symmetries. Furthermore, in \cite{Zan22} the authors only deal in detail the Melnikov function associated with the right subsystem and do not explain clearly how to ensure the simultaneous occurrence of zeros in the three Melnikov functions involved in the problem. It is worth noting that Theorem \ref{the:01} also deals with cases not yet addressed in the literature, i.e. the cases CSS and SSS.

The paper is organized as follows. In section \ref{sec:mel}, for prove the Theorem \ref{the:01}, we will obtain the expressions of the Melnikov functions associated with the system \eqref{eq:01}, whose zeros correspond to the limit cycles that bifurcate from the periodic annulus (see also \cite{ Xio20, Xio21} for more details about the Melnikov function). Our study is concentrated in the neighborhood of the double homoclinic loop which separates the periodic annulus of orbits that pass through three zones from the two periodic annulus of orbits that pass through only two zones. Thus, to estimate the zeros of the Melnikov functions we consider its expansions at the point corresponding to this orbit. What distinguishes our approach apart from others is that we provide, in the  Proposition \ref{the:coef}, a detailed analytical method to study the number of simultaneous zeros from Melnikov functions defined in two and three zones. Moreover, in order to simplifies the compute, we obtain a normal form to system $\eqref{eq:01}|_{\epsilon=0}$ in Section \ref{sec:NF}. Finally, in Section \ref{sec:Teo} we prove Theorem \ref{the:01}.


\section{Melnikov Function}\label{sec:mel}
Initially, we will introduce the Melnikov functions associated to system $\eqref{eq:01}$ which will be needed  to prove the main result of this paper. For this propose, consider the system $\eqref{eq:01}|_{\epsilon=0}$ satisfying hypothesis (H3), i.e. there are a open interval $J_0=(\alpha_0,\beta_0)$ and a periodic annulus consisting of a family of crossing periodic orbits $L^0_h$, with $h\in J_0$, such that each orbit of this family crosses the straight lines $x=\pm 1$ in four points, $A(h)=(1,a(h))$, $A_1(h)=(1,a_1(h))$, $A_2(h)=(-1,a_2(h))$ and $A_3(h)=(-1,a_3(h))$, with $a_1(h)<a(h)$ and $a_2(h)<a_3(h)$, thought the three zones with clockwise orientation, satisfying the following equations
\begin{equation*}\label{eq:050}
	\begin{aligned}
		& H^{\scriptscriptstyle R}(A(h))=H^{\scriptscriptstyle R}(A_1(h)), \\
		& H^{\scriptscriptstyle C}(A_1(h))=H^{\scriptscriptstyle C}(A_2(h)), \\
		& H^{\scriptscriptstyle L}(A_2(h))=H^{\scriptscriptstyle L}(A_3(h)), \\
		& H^{\scriptscriptstyle C}(A_3(h))=H^{\scriptscriptstyle C}(A(h)), 
	\end{aligned}
\end{equation*} 
and, for $h\in J_0$,
$$H^{\scriptscriptstyle R}_y(A(h))\,H^{\scriptscriptstyle R}_y(A_1(h))\,H^{\scriptscriptstyle L}_y(A_2(h))\,H^{\scriptscriptstyle L}_y(A_3(h))\ne 0,$$ 
$$H^{\scriptscriptstyle C}_y(A(h))\,H^{\scriptscriptstyle C}_y(A_1(h))\,H^{\scriptscriptstyle C}_y(A_2(h))\,H^{\scriptscriptstyle C}_y(A_3(h))\ne 0.$$
Therefore, for each $h\in J_0$, the crossing periodic orbit $L^0_h$ from system $\eqref{eq:01}|_{\epsilon=0}$ is given by $L^0_h=\widehat{AA_1}\cup\widehat{A_1A_2}\cup\widehat{A_2A_3}\cup\widehat{A_3A}$ (see Figure \ref{fig:01}).	

We also have, by hypothesis (H3), that there are two open interval $J_i=(\alpha_i,\beta_i)$, $i=1,2$, such that the unperturbed system $\eqref{eq:01}|_{\epsilon=0}$  has two periodic annulus consisting of two families of crossing periodic orbits $L_h^i$,  $h\in J_i$, $i=1,2$, such that for $i=1$, each orbit of this family crosses the straight lines $x=1$ in two points, $B(h)=(1,a(h))$ and $B_1(h)=(1,b_1(h))$, with $b_1(h)<b(h)$, and for $i=2$, each orbit of this family crosses the straight lines $x=-1$ in two points, $C(h)=(-1,c(h))$ and $C_1(h)=(-1,c_1(h))$, with $c_1(h)<c(h)$, where both orbits passing through the two zones with clockwise orientation, satisfying the following equations
\begin{equation*}\label{eq:051}
	\begin{aligned}
		& H^{\scriptscriptstyle R}(B(h))=H^{\scriptscriptstyle R}(B_1(h)), \\
		& H^{\scriptscriptstyle C}(B_1(h))=H^{\scriptscriptstyle C}(B(h)), \\
		& H^{\scriptscriptstyle C}(C(h))=H^{\scriptscriptstyle C}(C_1(h)), \\
		& H^{\scriptscriptstyle L}(C_1(h))=H^{\scriptscriptstyle L}(C(h)), 
	\end{aligned}
\end{equation*}
and, for $h\in J_i$, $i=1,2$,
\begin{equation*}
	\begin{aligned}
		& H^{\scriptscriptstyle R}_y(B(h))\,H^{\scriptscriptstyle R}_y(B_1(h))\,H^{\scriptscriptstyle L}_y(C(h))\,H^{\scriptscriptstyle L}_y(C_1(h))\ne 0,\\
		& H^{\scriptscriptstyle C}_y(B(h))\,H^{\scriptscriptstyle C}_y(B_1(h))\,H^{\scriptscriptstyle C}_y(C(h))\,H^{\scriptscriptstyle C}_y(C_1(h))\ne 0.
	\end{aligned}
\end{equation*}
Therefore, for each $h\in J_1$ (resp. $h\in J_2$), the crossing periodic orbit $L^1_h$ (resp. $L^2_h$) from system $\eqref{eq:01}|_{\epsilon=0}$ is given by $L^1_h=\widehat{BB_1}\cup\widehat{B_1B}$ (resp. $L^2_h=\widehat{CC_1}\cup\widehat{C_1C}$)  (see Figure \ref{fig:01}).

\begin{figure}[h]
	\begin{center}		
		\begin{overpic}[width=4in]{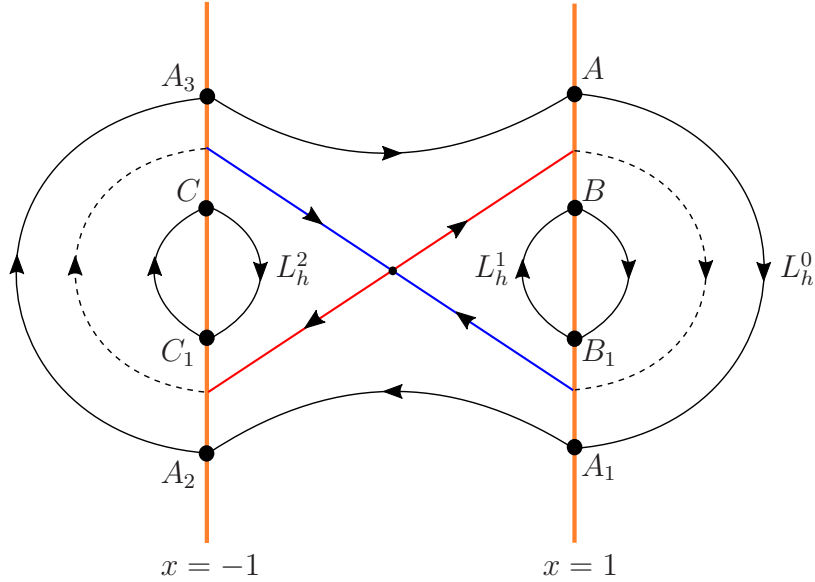}
			\put(70,-4) {$x=1$}
			\put(20,-4) {$x=-1$}
			\put(75,61) {$A$}
			\put(75,9) {$A_1$}
			\put(20,8) {$A_2$}
			\put(20,60) {$A_3$}
			\put(75,45) {$B$}
			\put(75,24) {$B_1$}
			\put(20,24) {$C_1$}
			\put(22,45) {$C$}
			\put(101,35) {$L_h^0$}
			\put(61,35) {$L_h^1$}
			\put(35,35) {$L_h^2$}
		\end{overpic}
	\end{center}
	\vspace{0.7cm}
	\caption{The crossing periodic orbits of system  $\eqref{eq:01}|_{\epsilon=0}$.}\label{fig:01}
\end{figure} 

Consider, for $h\in J_0$, the solution of right subsystem from \eqref{eq:01} starting at point $A(h)$. Let $A_{1\epsilon}(h)=(1,a_{1\epsilon}(h))$ be the first intersection point of this orbit with straight line $x=1$. Denote by $A_{2\epsilon}(h)=(-1,a_{2\epsilon}(h))$ the first intersection point of the orbit of central subsystem from \eqref{eq:01} starting at $A_{1\epsilon}(h)$ with straight line $x=-1$, $A_{3\epsilon}(h)=(-1,a_{3\epsilon}(h))$ the first intersection point of the orbit of left subsystem from \eqref{eq:01} starting at $A_{2\epsilon}(h)$ with straight line $x=-1$ and $A_{\epsilon}(h)=(1,a_{\epsilon}(h))$ the first intersection point of the orbit of central subsystem from \eqref{eq:01} starting at $A_{3\epsilon}(h)$ with straight line $x=1$ (see Figure \ref{fig:02}). Similarity, for $h\in J_1$ (resp. for $h\in J_2$), consider the solution of right (resp. central) subsystem from \eqref{eq:01} starting at point $B(h)$ (resp. at point $C(h)$). Let $B_{1\epsilon}(h)=(1,b_{1\epsilon}(h))$ (resp. $C_{1\epsilon}(h)=(-1,c_{1\epsilon}(h))$) be the first intersection point of this orbit with straight line $x=1$ (resp. $x=-1$) and  $B_{\epsilon}(h)=(1,b_{\epsilon}(h))$ (resp. $C_{\epsilon}(h)=(-1,c_{\epsilon}(h))$) the first intersection point of the orbit of central (resp. left) subsystem from \eqref{eq:01} starting at $B_{1\epsilon}(h)$ (resp. $C_{1\epsilon}(h)$) with straight line $x=1$ (resp. $x=-1$) (see Figure \ref{fig:02}). 
\begin{figure}[h]
	\begin{center}		
		\begin{overpic}[width=4in]{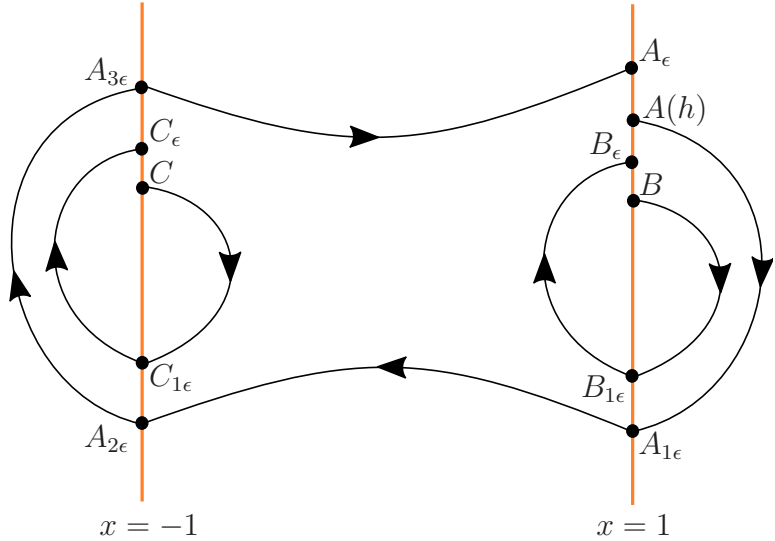}
			\put(77,-4) {$x=1$}
			\put(12,-4) {$x=-1$}
			\put(83,50.5) {$A(h)$}
			\put(82.5,6.5) {$A_{1\epsilon}$}
			\put(10,7.5) {$A_{2\epsilon}$}
			\put(10,55.5) {$A_{3\epsilon}$}
			\put(82.5,58) {$A_{\epsilon}$}
			\put(82.5,40.5) {$B$}
			\put(75,14) {$B_{1\epsilon}$}
			\put(76,45.8) {$B_{\epsilon}$}
			\put(18.5,42) {$C$}
			\put(18.5,15.5) {$C_{1\epsilon}$}
			\put(18.5,47.5) {$C_{\epsilon}$}
		\end{overpic}
	\end{center}
	\vspace{0.7cm}
	\caption{Poincaré map of system  \eqref{eq:01}.}\label{fig:02}
\end{figure} 
We define the Poincaré maps of piecewise system \eqref{eq:01} as follows, 
\begin{equation*}
	\begin{aligned}
		& H^{\scriptscriptstyle R}(A_{\epsilon}(h))-H^{\scriptscriptstyle R}(A(h))=\epsilon M_0(h)+\mathcal{O}(\epsilon^2),\quad \forall h\in J_0,\\
		& H^{\scriptscriptstyle R}(B_{\epsilon}(h))-H^{\scriptscriptstyle R}(B(h))=\epsilon M_1(h)+\mathcal{O}(\epsilon^2),\quad \forall h\in J_1,\\
		& H^{\scriptscriptstyle L}(C_{\epsilon}(h))-H^{\scriptscriptstyle L}(C(h))=\epsilon M_2(h)+\mathcal{O}(\epsilon^2),\quad \forall h\in J_2,
	\end{aligned}
\end{equation*}
where $M_0(h)$ and $M_i(h)$, $i=1,2$, are called the {\it first order Melnikov functions} associated to piecewise system \eqref{eq:01} in three and two zones, respectively. Then, following the steps of the proof of the Theorem 1.1 in \cite{Liu10} for the case in two zones and, doing the obvious adaptations, for the case of three zones, we can prove the following theorem (see also \cite{Xio21}).
\begin{theorem}\label{teo:mel}
	Consider  system \eqref{eq:01} with $0\leq \epsilon <<1$ and suppose that the unperturbed system $\eqref{eq:01}|_{\epsilon=0}$ satisfies the hypotheses $(\text{H}i)$, $i=1,2,3$. Then the first order Melnikov functions associated to system \eqref{eq:01} can be expressed as
	\begin{equation*}\label{eq:mel0}
		\begin{aligned}
			M_0(h) &  = \frac{H_y^{\scriptscriptstyle R}(A)}{H_y^{\scriptscriptstyle C}(A)} I_{\scriptscriptstyle C}^0 + \frac{H_y^{\scriptscriptstyle R}(A)H_y^{\scriptscriptstyle C}(A_3)}{H_y^{\scriptscriptstyle C}(A)H_y^{\scriptscriptstyle L}(A_3)} I_{\scriptscriptstyle L}^0 + \frac{H_y^{\scriptscriptstyle R}(A)H_y^{\scriptscriptstyle C}(A_3)H_y^{\scriptscriptstyle L}(A_2)}{H_y^{\scriptscriptstyle C}(A)H_y^{\scriptscriptstyle L}(A_3)H_y^{\scriptscriptstyle C}(A_2)} \bar{I}_{\scriptscriptstyle C}^0 \\
			&\quad  +  \frac{H_y^{\scriptscriptstyle R}(A)H_y^{\scriptscriptstyle C}(A_3)H_y^{\scriptscriptstyle L}(A_2)H_y^{\scriptscriptstyle C}(A_1)}{H_y^{\scriptscriptstyle C}(A)H_y^{\scriptscriptstyle L}(A_3)H_y^{\scriptscriptstyle C}(A_2)H_y^{\scriptscriptstyle R}(A_1)} I_{\scriptscriptstyle R}^0,\quad h\in J_0,
		\end{aligned}
	\end{equation*} 
	\begin{equation*}\label{eq:mel1}
		\begin{aligned} 
			M_1(h) & = \frac{H_y^{\scriptscriptstyle R}(B)}{H_y^{\scriptscriptstyle C}(B)}I^{\scriptscriptstyle 1}_{\scriptscriptstyle C} + \frac{H_y^{\scriptscriptstyle R}(B)H_y^{\scriptscriptstyle C}(B_1)}{H_y^{\scriptscriptstyle C}(B)H_y^{\scriptscriptstyle R}(B_1)} I^{\scriptscriptstyle 1}_{\scriptscriptstyle R}, \quad h\in J_1, 
		\end{aligned}
	\end{equation*}
and 
  	\begin{equation*}\label{eq:mel2}
  	\begin{aligned} 
  		M_2(h) & = \frac{H_y^{\scriptscriptstyle C}(C)}{H_y^{\scriptscriptstyle L}(C)}I^{\scriptscriptstyle 2}_{\scriptscriptstyle L} + \frac{H_y^{\scriptscriptstyle C}(C)H_y^{\scriptscriptstyle L}(C_1)}{H_y^{\scriptscriptstyle L}(C)H_y^{\scriptscriptstyle C}C_1)} I^{\scriptscriptstyle 2}_{\scriptscriptstyle C}, \quad h\in J_2, 
  	\end{aligned}
  \end{equation*}
	where
	$$
	I^{\scriptscriptstyle 0}_{\scriptscriptstyle C}= \int_{\widehat{A_3A}}g_{\scriptscriptstyle C}dx-f_{\scriptscriptstyle C}dy,\quad I^{\scriptscriptstyle 0}_{\scriptscriptstyle L}= \int_{\widehat{A_2A_3}}g_{\scriptscriptstyle L}dx-f_{\scriptscriptstyle L}dy,\quad \bar{I}^{\scriptscriptstyle 0}_{\scriptscriptstyle C}=\int_{\widehat{A_1A_2}}g_{\scriptscriptstyle C}dx-f_{\scriptscriptstyle C}dy,
	$$
	$$
	I^{\scriptscriptstyle 0}_{\scriptscriptstyle R}=\int_{\widehat{AA_1}}g_{\scriptscriptstyle C}dx-f_{\scriptscriptstyle C}dy,\quad I^{\scriptscriptstyle 1}_{\scriptscriptstyle C}=\int_{\widehat{B_1B}}g_{\scriptscriptstyle C}dx-f_{\scriptscriptstyle C}dy,\quad I^{\scriptscriptstyle 1}_{\scriptscriptstyle R}=\int_{\widehat{BB_1}}g_{\scriptscriptstyle R}dx-f_{\scriptscriptstyle R}dy,
	$$
	$$
	 I^{\scriptscriptstyle 2}_{\scriptscriptstyle L}=\int_{\widehat{C_1C}}g_{\scriptscriptstyle L}dx-f_{\scriptscriptstyle L}dy\quad\text{and}\quad I^{\scriptscriptstyle 2}_{\scriptscriptstyle C}=\int_{\widehat{CC_1}}g_{\scriptscriptstyle C}dx-f_{\scriptscriptstyle C}dy,
	$$
	Furthermore, if $M_i(h)$, for $i=0,1,2$, has a simple zero at $h^{*}$, then for $0< \epsilon <<1$, the system \eqref{eq:01} has a unique limit cycle near $L_{h^{*}}$. 
\end{theorem}


\section{Normal Form}\label{sec:NF}
 Now, we will do a continuous linear change of variables that keeps invariant the straight lines $x=\pm 1$, in order to decrease the number of parameters of system $\eqref{eq:01}|_{\epsilon=0}$. This change of variables is a homeomorphism and will be a topological equivalence between the systems. More precisely, following the steps of the proof of the Proposition 3 from \cite{Pes22b}, and doing the obvious adaptations, we have the follow result.

\begin{proposition}\label{fn:01}
Suppose that the central subsystem from $\eqref{eq:01}|_{\epsilon=0}$ has a saddle and the other two subsystems have saddles or centers. Then, after a linear change of variables and a rescaling of the independent variable, we can written the system $\eqref{eq:01}|_{\epsilon=0}$ with

\begin{equation}\label{nf:h}
	H(x,y)=\left\{\begin{array}{ll}\vspace{0.2cm}
		H^{\scriptscriptstyle L}(x,y)=\dfrac{b_{\scriptscriptstyle L}}{2}y^2-\dfrac{c_{\scriptscriptstyle L}}{2}x^2+a_{\scriptscriptstyle L}xy+a_{\scriptscriptstyle L}y-\beta_{\scriptscriptstyle L}x, \quad\quad x\leq -1, \\ \vspace{0.2cm}
		H^{\scriptscriptstyle C}(x,y)=\dfrac{y^2}{2}-\dfrac{x^2}{2} - \beta_{\scriptscriptstyle C}x, \quad\quad\quad\quad\quad\quad\quad\quad\quad\quad -1\leq x\leq 1, \\
		H^{\scriptscriptstyle R}(x,y)=\dfrac{b_{\scriptscriptstyle R}}{2}y^2-\dfrac{c_{\scriptscriptstyle R}}{2}x^2+a_{\scriptscriptstyle R}xy-a_{\scriptscriptstyle R}y-\beta_{\scriptscriptstyle R}x, \quad\quad x \geq 1. \\
	\end{array}
	\right.
\end{equation} 
\end{proposition}

\medskip

\begin{remark}
	Consider the system $\eqref{eq:01}|_{\epsilon=0}$ in its normal form, i.e. with $H(x,y)$ given by \eqref{nf:h}. In this paper, we restrict the hypothesis  (H1) a little more, more precisely we will assume that $\beta_{\scriptscriptstyle C}=0$, i.e. the singular point of the central subsystem from $\eqref{eq:01}|_{\epsilon=0}$ is at the origin. This extra restriction, apparently, does not change the number of limit cycles. Since, by the hypothesis (H1), we will always have a double homoclinic loop regardless of the position of the real saddle. However, it simplifies the expressions of Melnikov functions, which becomes the compute more easy. Therefore, in what follows, we will consider the system \eqref{eq:01} with $f(x,y)$ and $g(x,y)$ given by \eqref{eq:02} and \eqref{eq:03}, respectively, and $H(x,y)$ given by \eqref{nf:h} with $\beta_{\scriptscriptstyle C}=0$.
\end{remark}

\bigskip


In the follow proposition, we provide the open intervals $J_i$, $i=0,1,2$, where the first order Melnikov functions associated to system \eqref{eq:01} are define.


\begin{proposition}
	Consider the system \eqref{eq:01} assuming the hypotheses (Hi), $i=1,2,3$. 
	\begin{itemize}
		\item[{\rm (a)}] If the system $\eqref{eq:01}|_{\epsilon=0}$ is of type SSS or CSS, then  $J_0=(1,\tau_{\scriptscriptstyle RS})$ and $J_1=J_2=(0,1)$, where $\tau_{\scriptscriptstyle RS}=(a_{\scriptscriptstyle R}^2-b_{\scriptscriptstyle R}\beta_{\scriptscriptstyle R}-\omega_{\scriptscriptstyle RS}^2)/b_{\scriptscriptstyle R}\omega_{\scriptscriptstyle RS}$ with $\omega_{\scriptscriptstyle RS}=\sqrt{a^2_{\scriptscriptstyle R} + b_{\scriptscriptstyle R}c_{\scriptscriptstyle R}}$, and the periodic annulus are equivalents to one of the Figure \ref{fig:03sss2}--\ref{fig:03css2}.  
		\item[{\rm (b)}] If the system $\eqref{eq:01}|_{\epsilon=0}$ is of type CSC, then $J_0=(1,\infty)$ and $J_1=J_2=(0,1)$, with the periodic annulus equivalents to one of the Figure \ref{fig:04}.
	\end{itemize}
\end{proposition}
\begin{proof}
	Firstly, suppose that the system $\eqref{eq:01}|_{\epsilon=0}$ is of type SSS or CSS. Note that, if the saddles from the right or left subsystems are virtual or boundary then we have not periodic orbits passing through the three zones. Denote by $W^u_{\scriptscriptstyle L}$, $W^u_{\scriptscriptstyle C}$ and $W^u_{\scriptscriptstyle R}$ (resp. $W^s_{\scriptscriptstyle L}$, $W^s_{\scriptscriptstyle C}$ and $W^s_{\scriptscriptstyle R}$) the unstable (resp. stable) separatrices of the saddles of the left, central and right subsystems from $\eqref{eq:01}|_{\epsilon=0}$, respectively. Let $P_{\scriptscriptstyle L}^{i}=W^i_{\scriptscriptstyle L}\cap \{(-1,y):y\in\mathbb{R}\}$, $P_{\scriptscriptstyle R}^{i}=W^i_{\scriptscriptstyle R}\cap \{(1,y):y\in\mathbb{R}\}$, $P_{\scriptscriptstyle C}^{i}=W^i_{\scriptscriptstyle C}\cap \{(-1,y):y\in\mathbb{R}\}$ and $\tilde{P}_{\scriptscriptstyle C}^{i}=W^i_{\scriptscriptstyle C}\cap \{(1,y):y\in\mathbb{R}\}$, $i=u,s$. After some compute, it is possible to show that
	$$P_{\scriptscriptstyle L}^{u}=(-1,\tau_{\scriptscriptstyle LS}),\quad P_{\scriptscriptstyle L}^{s}=(-1,-\tau_{\scriptscriptstyle LS}),\quad P_{\scriptscriptstyle R}^{u}=(1,-\tau_{\scriptscriptstyle RS}),\quad P_{\scriptscriptstyle R}^{s}=(1,\tau_{\scriptscriptstyle RS}),$$
	$$P_{\scriptscriptstyle C}^{u}=(-1,-1),\quad P_{\scriptscriptstyle C}^{s}=(-1,1),\quad \tilde{P}_{\scriptscriptstyle C}^{u}=(1,1),\quad \tilde{P}_{\scriptscriptstyle C}^{s}=(1,-1),$$
	where $\tau_{\scriptscriptstyle RS}=(a_{\scriptscriptstyle R}^2-b_{\scriptscriptstyle R}\beta_{\scriptscriptstyle R}-\omega_{\scriptscriptstyle RS}^2)/b_{\scriptscriptstyle R}\omega_{\scriptscriptstyle RS}$, $\tau_{\scriptscriptstyle LS}=(a_{\scriptscriptstyle L}^2+b_{\scriptscriptstyle L}\beta_{\scriptscriptstyle L}-\omega_{\scriptscriptstyle LS}^2)/b_{\scriptscriptstyle L}\omega_{\scriptscriptstyle LS}$, $\omega_{\scriptscriptstyle RS}=\sqrt{a^2_{\scriptscriptstyle R} + b_{\scriptscriptstyle R}c_{\scriptscriptstyle R}}$ and $\omega_{\scriptscriptstyle LS}=\sqrt{a^2_{\scriptscriptstyle L} + b_{\scriptscriptstyle L}c_{\scriptscriptstyle L}}$. Moreover, we have a symmetry between the points $P_{\scriptscriptstyle L}^{u}$ and $P_{\scriptscriptstyle L}^{s}$ (resp. $P_{\scriptscriptstyle R}^{u}$ and $P_{\scriptscriptstyle R}^{s}$) with respect to $x$-axis.  Define by $\tau$ the smallest ordinate value between the points $P_{\scriptscriptstyle R}^{s}$ and $P_{\scriptscriptstyle L}^{u}$, i.e. $\tau= \min\{\tau_{\scriptscriptstyle LS},\tau_{\scriptscriptstyle RS}\}$.
	Then, less than one reflection around the $y$-axis, we can assuming that $\tau=\tau_{\scriptscriptstyle RS}$.

\begin{figure}[h]\tiny
	\begin{center}		
		\begin{overpic}[width=4.5in]{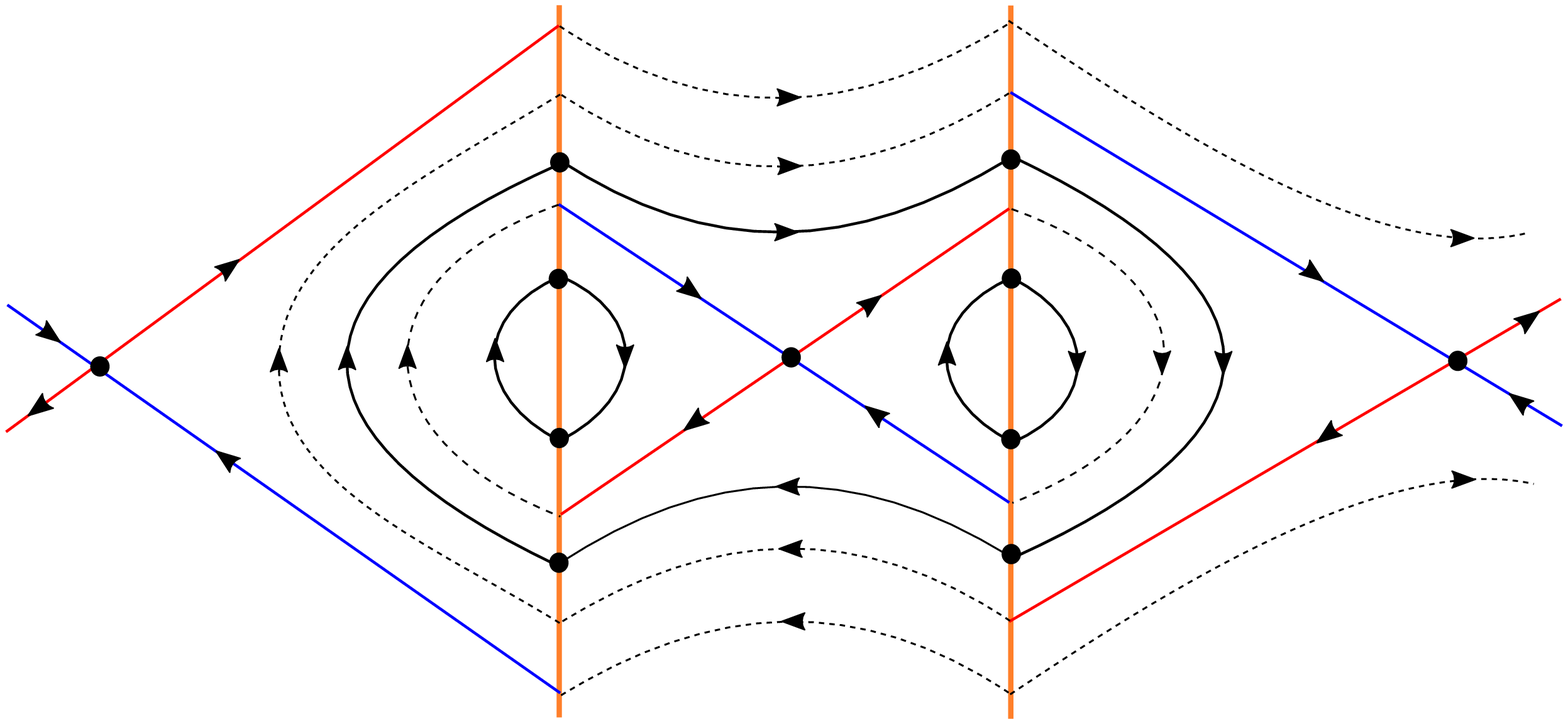}
			\put(61,-4) {$x=1$}
			\put(31,-4) {$x=-1$}
			\put(65,36.5) {$A$}
			\put(65,9) {$A_{1}$}
			\put(65,29) {$B$}
			\put(65,16) {$B_{1}$}
			\put(65,40.5) {$P_{\scriptscriptstyle R}^{s}$}
			\put(65,5) {$P_{\scriptscriptstyle R}^{u}$}
			\put(32,8.5) {$A_2$}
			\put(31.5,36) {$A_3$}
			\put(33,28.5) {$C$}
			\put(32,16) {$C_{1}$}
			\put(31.5,45) {$P_{\scriptscriptstyle L}^{u}$}
			\put(31.5,0) {$P_{\scriptscriptstyle L}^{s}$}									
		\end{overpic}
	\end{center}
	\vspace{0.7cm}
	\caption{Phase portrait of system $\eqref{eq:01}|_{\epsilon=0}$ of the type SSS with $\tau_{\scriptscriptstyle RS}\ne\tau_{\scriptscriptstyle LS}$.}\label{fig:03sss2}
\end{figure} 

\begin{figure}[h]\tiny
	\begin{center}		
		\begin{overpic}[width=4.5in]{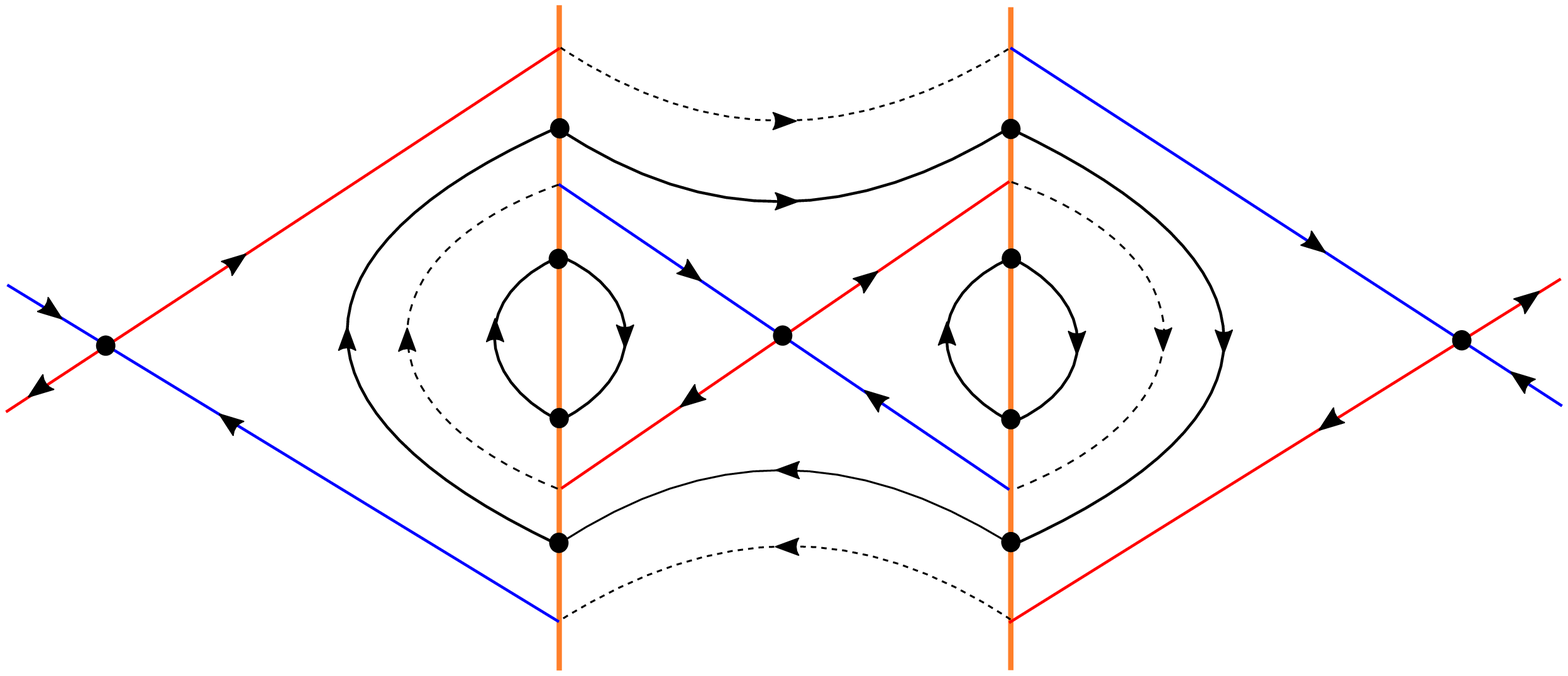}
			\put(61,-3) {$x=1$}
			\put(31,-3) {$x=-1$}
			\put(65,35.5) {$A$}
			\put(65,6.6) {$A_{1}$}
			\put(65,27) {$B$}
			\put(65,14.5) {$B_{1}$}
			\put(65,40) {$P_{\scriptscriptstyle R}^{s}$}
			\put(65,1.5) {$P_{\scriptscriptstyle R}^{u}$}
			\put(32,6.6) {$A_2$}
			\put(32,35.5) {$A_3$}
			\put(33,27) {$C$}
			\put(32,14.5) {$C_{1}$}
			\put(31.5,40.5) {$P_{\scriptscriptstyle L}^{u}$}
			\put(31.5,1.5) {$P_{\scriptscriptstyle L}^{s}$}									
		\end{overpic}
	\end{center}
	\vspace{0.7cm}
	\caption{Phase portrait of system $\eqref{eq:01}|_{\epsilon=0}$ of the type SSS with $\tau_{\scriptscriptstyle RS}=\tau_{\scriptscriptstyle LS}$.}\label{fig:03sss1}
\end{figure} 

\begin{figure}[h]\tiny
	\begin{center}		
		\begin{overpic}[width=4in]{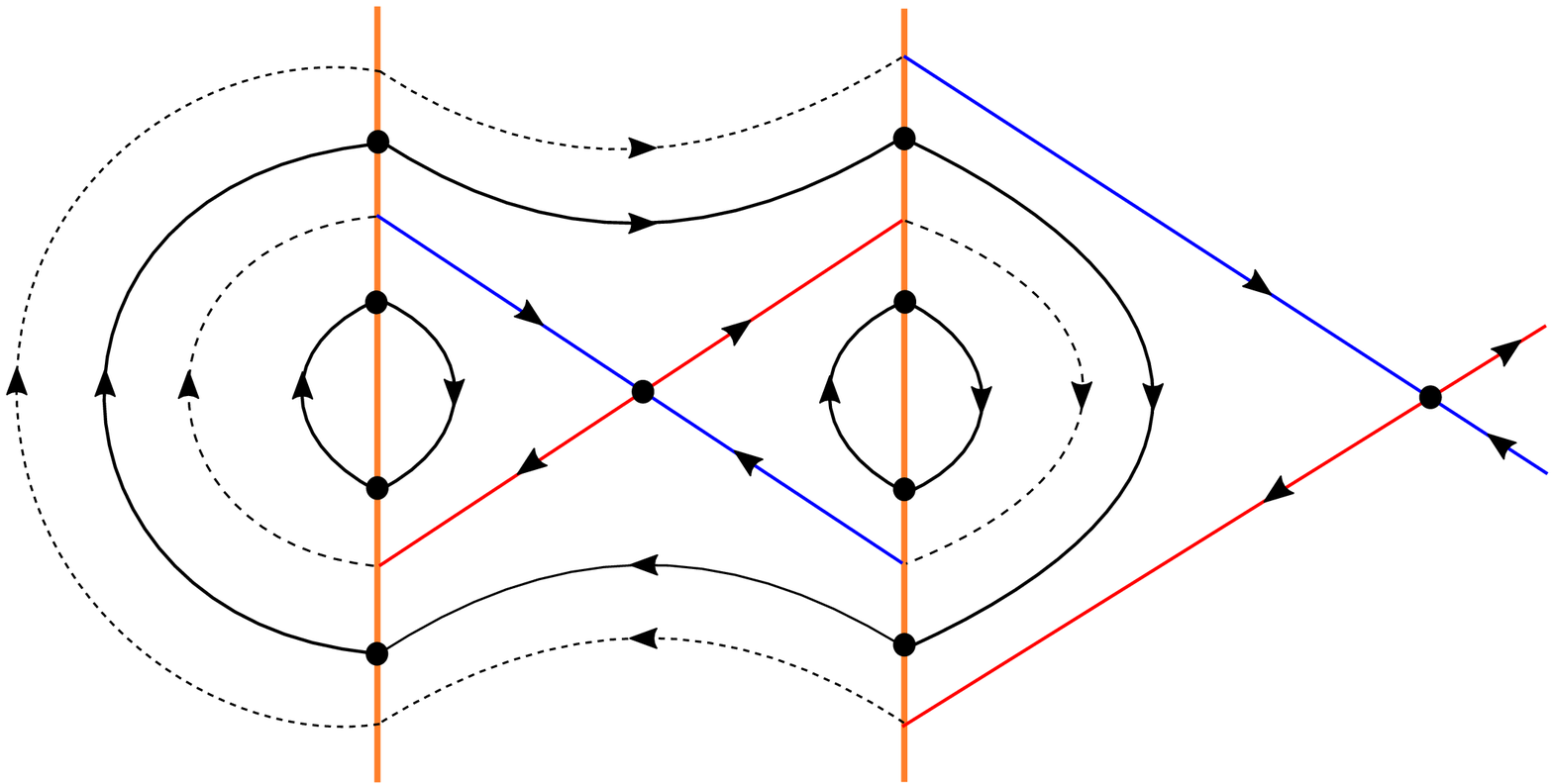}
			\put(55,-4) {$x=1$}
			\put(20,-4) {$x=-1$}
			\put(59,42) {$A$}
			\put(58.5,7) {$A_{1}$}
			\put(59,31.5) {$B$}
			\put(59,17) {$B_{1}$}
			\put(59,47) {$P_{\scriptscriptstyle R}^{s}$}
			\put(59,2) {$P_{\scriptscriptstyle R}^{u}$}
			\put(19,6) {$A_2$}
			\put(19,42) {$A_3$}
			\put(21,31) {$C$}
			\put(20,17) {$C_{1}$}								
		\end{overpic}
	\end{center}
	\vspace{0.7cm}
	\caption{Phase portrait of system $\eqref{eq:01}|_{\epsilon=0}$ of the type CSS with a virtual center in the left subsystem.}\label{fig:03css1}
\end{figure} 

\begin{figure}[h]\tiny
	\begin{center}		
		\begin{overpic}[width=4in]{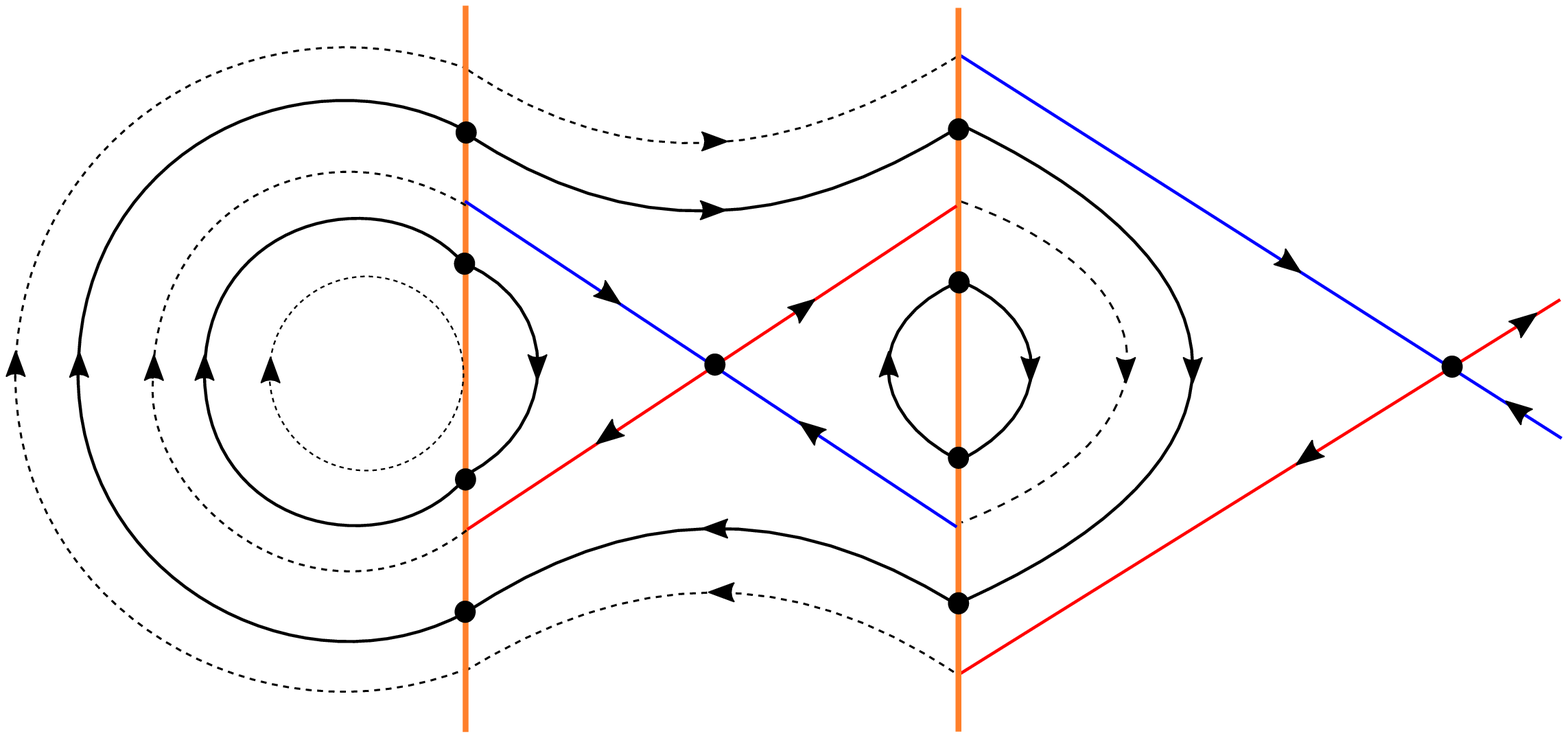}
			\put(58,-4) {$x=1$}
			\put(25,-4) {$x=-1$}
			\put(62,39) {$A$}
			\put(61.5,7) {$A_{1}$}
			\put(62,29.5) {$B$}
			\put(62,16) {$B_{1}$}
			\put(62,44) {$P_{\scriptscriptstyle R}^{s}$}
			\put(62,2) {$P_{\scriptscriptstyle R}^{u}$}
			\put(25,8) {$A_2$}
			\put(25.5,37) {$A_3$}
			\put(26.5,28.5) {$C$}
			\put(25,15.5) {$C_{1}$}								
		\end{overpic}
	\end{center}
	\vspace{0.7cm}
	\caption{Phase portrait of system $\eqref{eq:01}|_{\epsilon=0}$ of the type CSS with a real center in the left subsystem.}\label{fig:03css2}
\end{figure} 
Note that, if system $\eqref{eq:01}|_{\epsilon=0}$ is of the type SSS and the ordinates of the points $P_{\scriptscriptstyle R}^{s}$ and $P_{\scriptscriptstyle L}^{u}$ are distinct, i.e. $\tau_{\scriptscriptstyle RS}\ne \tau_{\scriptscriptstyle LS}$ (see Figure \ref{fig:03sss2})  or if system $\eqref{eq:01}|_{\epsilon=0}$ is of the type CSS (see Figure \ref{fig:03css1}--\ref{fig:03css2}), then we have a homoclinic loop passing through the points $P_{\scriptscriptstyle R}^{s}$ and $P_{\scriptscriptstyle R}^{u}$. Otherwise, if system $\eqref{eq:01}|_{\epsilon=0}$ is of the type SSS and the points $P_{\scriptscriptstyle R}^{s}$ and $P_{\scriptscriptstyle L}^{u}$ have the same ordinate, i.e. $\tau_{\scriptscriptstyle RS}= \tau_{\scriptscriptstyle LS}$ (see Figure \ref{fig:03sss1}), then we have a hetoclinic orbit passing through the points $P_{\scriptscriptstyle R}^{s}$, $P_{\scriptscriptstyle R}^{u}$, $P_{\scriptscriptstyle L}^{s}$ and $P_{\scriptscriptstyle L}^{u}$. Moreover, the central subsystem from $\eqref{eq:01}|_{\epsilon=0}$ have two tangent points $P_{\scriptscriptstyle R}=(1,0)$ and $P_{\scriptscriptstyle L}=(-1,0)$.  The Figures \ref{fig:03sss2}--\ref{fig:03css2} shows the possibles phase portraits of the system $\eqref{eq:01}|_{\epsilon=0}$ of the type SSS and CSS.

Consider a initial point of form $A(h)=(1,h)$, with $h\in J_0=(1,\tau_{\scriptscriptstyle RS})$. By the hypothesis (H3), the system $\eqref{eq:01}|_{\epsilon=0}$ has a family of crossing periodic orbits that intersects the straight lines $x=\pm1$ at four points, $A(h)$, $A_1(h)=(1,a_1(h))$, with $a_1(h)<h$, and $A_2(h)=(-1,a_2(h))$, $A_3(h)=(-1,a_3(h))$, with $a_2(h)<a_3(h)$, satisfying
\begin{equation*}
	\begin{aligned}
		& H^{\scriptscriptstyle R}(A(h))=H^{\scriptscriptstyle R}(A_1(h)), \\
		& H^{\scriptscriptstyle C}(A_1(h))=H^{\scriptscriptstyle C}(A_2(h)), \\
		& H^{\scriptscriptstyle L}(A_2(h))=H^{\scriptscriptstyle L}(A_3(h)), \\
		& H^{\scriptscriptstyle C}(A_3(h))=H^{\scriptscriptstyle C}(A(h)), 
	\end{aligned}
\end{equation*}
where $H^{\scriptscriptstyle R}$,  $H^{\scriptscriptstyle C}$ and  $H^{\scriptscriptstyle L}$ are given by the normal form from Proposition \ref{fn:01}. More precisely, we have the equations
\begin{equation*}
	\begin{aligned}
		& \frac{b_{\scriptscriptstyle R}}{2}(h-a_1(h))(h+a_1(h))=0, \\
		& \frac{1}{2}(a_1(h)-a_2(h))(a_1(h)+a_2(h))=0, \\
		& \frac{b_{\scriptscriptstyle L}}{2}(a_2(h)-a_3(h))(a_2(h)+a_3(h))=0, \\
		& \frac{1}{2}(a_3(h)-h)(a_3(h)+h)=0. 
	\end{aligned}
\end{equation*}
As $a_1(h)<h$, $a_2(h)<a_3(h)$, $b_{\scriptscriptstyle R}>0$ and $b_{\scriptscriptstyle L}>0$, the only solution of system above is $a_1(h)=a_2(h)=-h$ and $a_3(h)=h$, i.e. we have the four points given by $A(h)=(1,h)$, $A_1(h)=(1,-h)$, $A_2(h)=(-1,-h)$ and $A_3(h)=(-1,h)$. Moreover, for each $h\in J_0$, system  $\eqref{eq:01}|_{\epsilon=0}$ has a periodic orbit $L^0_h=\widehat{AA_1}\cup\widehat{A_1A_2}\cup\widehat{A_2A_3}\cup\widehat{A_3A}$ passing through these points. 

Similarly the previous case, consider two initial points of the form $B(h)=(1,h)$, with $h\in J_1=(0,1)$, and $C(h)=(-1,h)$, with $h\in J_2=(0,1)$. By hypothesis (H3), the system $\eqref{eq:01}|_{\epsilon=0}$ has two family of crossing periodic orbits passing through the two zones with clockwise orientation. The first one,  intersects the straight line $x=1$ at two points, $B(h)$ and $B_1(h)=(1,b_1(h))$, with $b_1(h)<h$. The second one, intersects the straight line $x=-1$ at two points, $C(h)$ and $C_1(h)=(1,c_1(h))$, with $c_1(h)<h$. Moreover, these family of crossing periodic orbits satisfying the following equations
	\begin{equation*}
	\begin{aligned}
		& H^{\scriptscriptstyle R}(B(h))=H^{\scriptscriptstyle R}(B_1(h)), \\
		& H^{\scriptscriptstyle C}(B_1(h))=H^{\scriptscriptstyle C}(B(h)), \\
		& H^{\scriptscriptstyle C}(C(h))=H^{\scriptscriptstyle C}(C_1(h)), \\
		& H^{\scriptscriptstyle L}(C_1(h))=H^{\scriptscriptstyle L}(C(h)).
	\end{aligned}
\end{equation*}
	More precisely, we have the equations 
\begin{equation*}
	\begin{aligned}
		& \frac{b_{\scriptscriptstyle R}}{2}(h-b_1(h))(h+b_1(h))=0, \\
		& \frac{1}{2}(b_1(h)-h)(b_1(h)+h)=0,\\
		& \frac{1}{2}(h-c_1(h))(h+c_1(h))=0, \\
		& \frac{b_{\scriptscriptstyle L}}{2}(c_1(h)-h)(c_1(h)+h)=0.
	\end{aligned}
\end{equation*}
As $b_1(h)<h$, $c_1(h)<h$, $b_{\scriptscriptstyle R}>0$ and $b_{\scriptscriptstyle L}>0$, the only solution of system above is $b_1(h)=c_1(h)=-h$, i.e. we have the four points given by $B(h)=(1,h)$, $B_1(h)=(1,-h)$, $C(h)=(-1,h)$ and $C_1(h)=(-1,-h)$. Moreover, system  $\eqref{eq:01}|_{\epsilon=0}$ has two periodic orbit $L^i_h$,  with $h\in J_i$, for $i=1,2$, passing through these points. If $h\in[\tau_{\scriptscriptstyle RS},\infty)$ then the orbit of the system $\eqref{eq:01}|_{\epsilon=0}$ with initial condition in $A(h)$ do not return to straight line $x=1$ to positive times, i.e. the system $\eqref{eq:01}|_{\epsilon=0}$ has no periodic orbit passing thought the point $A(h)$. Therefore, if $h\in(1,\tau_{\scriptscriptstyle RS})$ the system $\eqref{eq:01}|_{\epsilon=0}$ has a periodic annulus, formed by the periodic orbits $L^0_h$, limited by one double homoclinic loop at $h=1$ and either a homoclinic loop (see Figures \ref{fig:03sss2}--\ref{fig:03css2}) or two heteroclinic orbits (see Fig. \ref{fig:03sss1}) at $h=\tau_{\scriptscriptstyle RS}$. Now, if $h\in(0,1)$ the system $\eqref{eq:01}|_{\epsilon=0}$ has two periodic annulus, formed by the periodic orbits $L^i_h$, $i=1,2$, and limited by homoclinic loops. Therefore, item (a) is proven.

To prove item (b), suppose that the system $\eqref{eq:01}|_{\epsilon=0}$ is of type CSC. As in the previous cases, for each $h\in(1,\infty)$ we have a periodic orbit $L^0_h$ passing through the points $A(h)=(1,h)$, $A_1(h)=(1,-h)$, $A_2(h)=(-1,-h)$ and $A_3(h)=(-1,h)$. Therefore, the system $\eqref{eq:01}|_{\epsilon=0}$ has a continuum of periodic orbits formed by the periodic orbits $L^0_h$, with $h\in(1,\infty)$, limited by one double homoclinic loop (see Fig. \ref{fig:04} (a)--(c)). Now, if $h\in(0,1)$ the system $\eqref{eq:01}|_{\epsilon=0}$ has two periodic annulus, formed by the periodic orbits $L^i_h$, $i=1,2$, with $L^1_h$ passing through the points $B(h)=(1,h)$, $B_1(h)=(1,-h)$ and $L^2_h$ passing through the points $C(h)=(-1,h)$, $C_1(h)=(-1,-h)$, and limited by the homoclinic loops.  The Figure \ref{fig:04} shows the possibles phase portraits of the system $\eqref{eq:01}|_{\epsilon=0}$ of the type CSC.
\begin{figure}[h]\tiny
	\begin{center}		
		\begin{overpic}[width=6in]{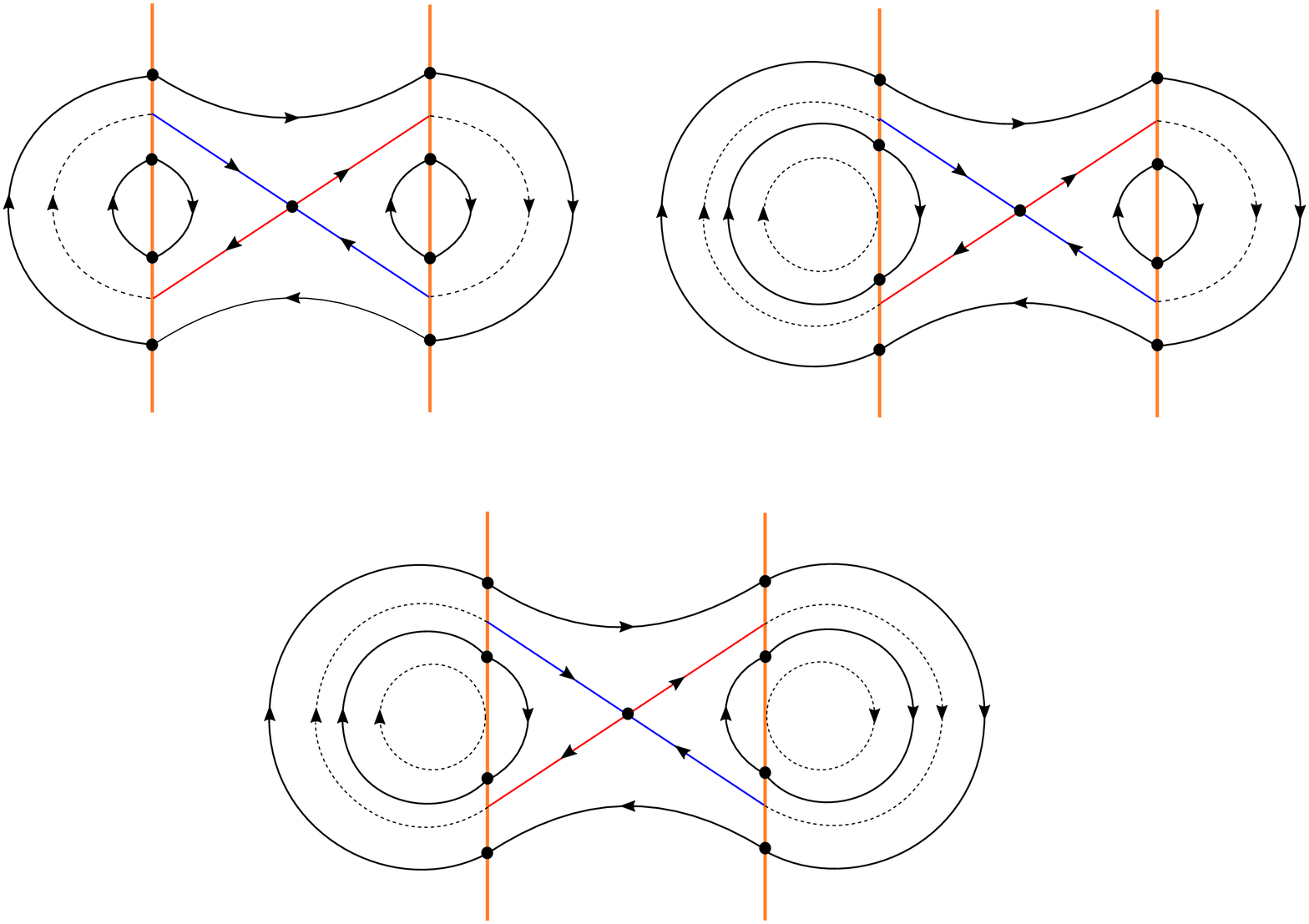}
			\put(86,37) {$x=1$}
			\put(64,37) {$x=-1$}
			\put(89,65) {$A$}
			\put(89,42.5) {$A_{1}$}
			\put(89,58.5) {$B$}
			\put(89.2,49) {$B_{1}$}
			\put(64.5,66) {$A_3$}
			\put(64.5,41) {$A_2$}
			\put(65,58) {$C$}
			\put(64.5,49) {$C_1$}
			\put(77,37) {\normalsize (b)}
			\put(30.5,37) {$x=1$}
			\put(8,37) {$x=-1$}
			\put(33.5,65.5) {$A$}
			\put(33.5,42.7) {$A_1$}
			\put(8.5,42.7) {$A_2$}
			\put(8.5,65.5) {$A_3$}
			\put(33.5,58.5) {$B$}
			\put(33.5,49.5) {$B_{1}$}
			\put(9.5,58.5) {$C$}
			\put(9,49.5) {$C_1$}
			\put(20,37) {\normalsize (a)}
			\put(56,-2) {$x=1$}
			\put(34,-2) {$x=-1$}
			\put(58.7,27) {$A$}
			\put(59,2.7) {$A_1$}
			\put(59,19.3) {$B$}
			\put(59,10.7) {$B_1$}
			\put(34.5,2.7) {$A_2$}
			\put(35.5,27) {$A$}
			\put(35,19.3) {$C$}
			\put(34.5,10.7) {$C_1$}
			\put(46.5,-2) {\normalsize (c)}
		\end{overpic}
	\end{center}
	\vspace{0.7cm}
	\caption{Phase portrait of system $\eqref{eq:01}|_{\epsilon=0}$ of the type CSC when (a) left and right subsystems have virtual centers; (b) left subsystem has a real center and right subsystems has a virtual center; (c) left and right subsystems have real centers.}\label{fig:04}
\end{figure}
\end{proof}

As, for $h\in J_0$, $A(h)=(1,h)$, $A_1(h)=(1,-h)$, $A_2(h)=(-1,-h)$,  $A_3(h)=(-1,h)$ and, for $h\in J_i$, $i=1,2$, $B(h)=(1,h)$, $B_1(h)=(1,-h)$,  $C(h)=(-1,h)$ and $C_1(h)=(-1,-h)$, we have the follow immediate corollary. 
\begin{corollary}\label{col:mel}
	For $h\in J_0$, we have that
	$$
	\frac{H_y^{\scriptscriptstyle R}(A)}{H_y^{\scriptscriptstyle C}(A)}=\frac{H_y^{\scriptscriptstyle R}(A)H_y^{\scriptscriptstyle C}(A_3)H_y^{\scriptscriptstyle L}(A_2)}{H_y^{\scriptscriptstyle C}(A)H_y^{\scriptscriptstyle L}(A_3)H_y^{\scriptscriptstyle C}(A_2)}=b_{\scriptscriptstyle R},\quad \frac{H_y^{\scriptscriptstyle R}(A)H_y^{\scriptscriptstyle C}(A_3)}{H_y^{\scriptscriptstyle C}(A)H_y^{\scriptscriptstyle L}(A_3)}=\frac{b_{\scriptscriptstyle R}}{b_{\scriptscriptstyle L}},$$
	$$\frac{H_y^{\scriptscriptstyle R}(A)H_y^{\scriptscriptstyle C}(A_3)H_y^{\scriptscriptstyle L}(A_2)H_y^{\scriptscriptstyle C}(A_1)}{H_y^{\scriptscriptstyle C}(A)H_y^{\scriptscriptstyle L}(A_3)H_y^{\scriptscriptstyle C}(A_2)H_y^{\scriptscriptstyle R}(A_1)}=1,
	$$
	and for $h\in J_i$, $i=1,2$, we have that
	$$\frac{H_y^{\scriptscriptstyle R}(B)}{H_y^{\scriptscriptstyle C}(B)}=b_{\scriptscriptstyle R},\quad \frac{H_y^{\scriptscriptstyle C}(C)}{H_y^{\scriptscriptstyle L}(C)}=\frac{1}{b_{\scriptscriptstyle L}},\quad\frac{H_y^{\scriptscriptstyle R}(B)H_y^{\scriptscriptstyle C}(B_1)}{H_y^{\scriptscriptstyle C}(B)H_y^{\scriptscriptstyle R}(B_1)}=\frac{H_y^{\scriptscriptstyle C}(C)H_y^{\scriptscriptstyle L}(C_1)}{H_y^{\scriptscriptstyle L}(C)H_y^{\scriptscriptstyle C}C_1)}=1.$$
	Then, the first order Melnikov functions associated to system \eqref{eq:01} can be rewritten as
\begin{eqnarray}
	&& M_0(h) = b_{\scriptscriptstyle R}\int_{\widehat{A_3A}}g_{\scriptscriptstyle C}dx-f_{\scriptscriptstyle C}dy+\frac{b_{\scriptscriptstyle R}}{b_{\scriptscriptstyle L}} \int_{\widehat{A_2A_3}}g_{\scriptscriptstyle L}dx-f_{\scriptscriptstyle L}dy \label{mel:M0}\\
	&& \quad\quad\quad \,\,\,\, + b_{\scriptscriptstyle R}\int_{\widehat{A_1A_2}}g_{\scriptscriptstyle C}dx-f_{\scriptscriptstyle C}dy+\int_{\widehat{AA_1}}g_{\scriptscriptstyle R}dx-f_{\scriptscriptstyle R}dy,\nonumber\\
	&& M_1(h) = b_{\scriptscriptstyle R}\int_{\widehat{B_1B}}g_{\scriptscriptstyle C}dx-f_{\scriptscriptstyle C}dy+\int_{\widehat{BB_1}}g_{\scriptscriptstyle R}dx-f_{\scriptscriptstyle R}dy,\label{mel:M1}\\
	&& M_2(h) = \frac{1}{b_{\scriptscriptstyle L}}\int_{\widehat{C_1C}}g_{\scriptscriptstyle L}dx-f_{\scriptscriptstyle L}dy+\int_{\widehat{CC_1}}g_{\scriptscriptstyle C}dx-f_{\scriptscriptstyle C}dy\label{mel:M2}.
\end{eqnarray}	
\end{corollary}

\section{Proof of Theorem \ref{the:01}}\label{sec:Teo}
In order to prove the Theorem \ref{the:01}, we will simplify the expressions of the first order Melnikov functions associated with system \eqref{eq:01} when the system $\eqref{eq:01}|_{\epsilon=0}$ is of the type SSS, CSS and CSC. For this, we define the follows functions: 
\begin{equation}\label{eq:func}
	\begin{aligned}
		& f_0(h)  =  h,  \quad h\in(0,\infty),\\
		& f_{\scriptscriptstyle C}^{0}(h)  = (h^2-1)\log\bigg(\frac{h+1}{h-1}\bigg),\quad h\in(1,\infty), \\
		& f_{\scriptscriptstyle C}(h)  = (h^2-1)\log\bigg(\frac{h+1}{1-h}\bigg),\quad h\in(0,1), \\
		& f_{\scriptscriptstyle RC}(h)  = (h^2+\tau_{\scriptscriptstyle RC}^2)\bigg(2\pi\mu_1+(-1)^{\mu_1}\arccos\bigg(\frac{\tau_{\scriptscriptstyle RC}^2-h^2}{h^2+\tau_{\scriptscriptstyle RC}^2}\bigg)\bigg),\quad h\in(0,\infty),\\
		& f_{\scriptscriptstyle LC}(h)  = (h^2+\tau_{\scriptscriptstyle LC}^2)\bigg(2\pi\mu_2+(-1)^{\mu_2}\arccos\bigg(\frac{\tau_{\scriptscriptstyle LC}^2-h^2}{h^2+\tau_{\scriptscriptstyle LC}^2}\bigg)\bigg),\quad h\in(0,\infty),\\
		& f_{\scriptscriptstyle RS}(h)  = (h^2-\tau_{\scriptscriptstyle RS}^2)\log\bigg(\frac{h+\tau_{\scriptscriptstyle RS}}{\tau_{\scriptscriptstyle RS}-h}\bigg),\quad h\in(0,1)\cup(1,\tau_{\scriptscriptstyle RS}),\\
		& f_{\scriptscriptstyle LS}(h)  = (h^2-\tau_{\scriptscriptstyle LS}^2)\log\bigg(\frac{h+\tau_{\scriptscriptstyle LS}}{\tau_{\scriptscriptstyle LS}-h}\bigg),\quad h\in(0,1)\cup(1,\tau_{\scriptscriptstyle RS}),
	\end{aligned}
\end{equation}	
with $\tau_{\scriptscriptstyle RC}=(a_{\scriptscriptstyle R}^2-b_{\scriptscriptstyle R}\beta_{\scriptscriptstyle R}+\omega_{\scriptscriptstyle RC}^2)/b_{\scriptscriptstyle R}\omega_{\scriptscriptstyle RC}$, $\tau_{\scriptscriptstyle LC}=(a_{\scriptscriptstyle L}^2+b_{\scriptscriptstyle L}\beta_{\scriptscriptstyle L}+\omega_{\scriptscriptstyle LC}^2)/b_{\scriptscriptstyle L}\omega_{\scriptscriptstyle LC}$, $\tau_{\scriptscriptstyle RS}=(a_{\scriptscriptstyle R}^2-b_{\scriptscriptstyle R}\beta_{\scriptscriptstyle R}-\omega_{\scriptscriptstyle RS}^2)/b_{\scriptscriptstyle R}\omega_{\scriptscriptstyle RS}$, $\tau_{\scriptscriptstyle LS}=(a_{\scriptscriptstyle L}^2+b_{\scriptscriptstyle L}\beta_{\scriptscriptstyle L}-\omega_{\scriptscriptstyle LS}^2)/b_{\scriptscriptstyle L}\omega_{\scriptscriptstyle LS}$, $\omega_{i{\scriptscriptstyle C}}=\sqrt{-a^2_{\scriptscriptstyle i} - b_{\scriptscriptstyle i}c_{\scriptscriptstyle i}}$ and $\omega_{i{\scriptscriptstyle S}}=\sqrt{a^2_{\scriptscriptstyle i} + b_{\scriptscriptstyle i}c_{\scriptscriptstyle i}}$, for $i=L,R$.

\begin{theorem}\label{theo:sss}
	Suppose that system $\eqref{eq:01}|_{\epsilon=0}$ is of the type SSS. Then the first order Melnikov functions given by  \eqref{mel:M0}, \eqref{mel:M1} and \eqref{mel:M2} associated with system \eqref{eq:01} can be expressed, when $\tau_{\scriptscriptstyle LS}\ne\tau_{\scriptscriptstyle RS}$, as	
	\begin{eqnarray}
		&& M_0(h)=k_0^0f_0(h)+k_0^1f_{\scriptscriptstyle C}^{0}(h)+k_0^2f_{\scriptscriptstyle RS}(h)+k_0^3f_{\scriptscriptstyle LS}(h),\quad h\in(1,\tau_{\scriptscriptstyle RS}),\label{mel:sss00}\\
		&& M_1(h)=k_1^0f_0(h)+k_1^1f_{\scriptscriptstyle C}(h)+k_1^2f_{\scriptscriptstyle RS}(h),\quad\quad\quad\quad\quad\quad h\in(0,1),\label{mel:sss10}\\
		&& M_2(h)=k_2^0f_0(h)+k_2^1f_{\scriptscriptstyle C}(h)+k_2^2f_{\scriptscriptstyle LS}(h),\quad\quad\quad\quad\quad\quad h\in(0,1)\label{mel:sss20},
	\end{eqnarray}
and, when $\tau_{\scriptscriptstyle LS}=\tau_{\scriptscriptstyle RS}$, as
	\begin{eqnarray}
	&& M_0(h)=\tilde{k}_0^0f_0(h)+\tilde{k}_0^1f_{\scriptscriptstyle C}^{0}(h)+\tilde{k}_0^2f_{\scriptscriptstyle RS}(h),\quad h\in(1,\tau_{\scriptscriptstyle RS}),\label{mel:sss01}\\
	&& M_1(h)=\tilde{k}_1^0f_0(h)+\tilde{k}_1^1f_{\scriptscriptstyle C}(h)+\tilde{k}_1^2f_{\scriptscriptstyle RS}(h),\quad h\in(0,1),\label{mel:sss11}\\
	&& M_2(h)=\tilde{k}_2^0f_0(h)+\tilde{k}_2^1f_{\scriptscriptstyle C}(h)+\tilde{k}_2^2f_{\scriptscriptstyle RS}(h),\quad h\in(0,1)\label{mel:sss21},
\end{eqnarray}
 where the functions $f_0(h)$, $f_{\scriptscriptstyle C}^0(h)$, $f_{\scriptscriptstyle C}(h)$, $f_{\scriptscriptstyle RS}(h)$ and $f_{\scriptscriptstyle LS}(h)$ are the ones defined in \eqref{eq:func}. Here the coefficients $k_i^j$ and $\tilde{k}_i^j$, for $i=0,1,2$ and $j=0,1,2,3$, depend on the parameters of system \eqref{eq:01}.
\end{theorem}
\begin{proof}
	 	Firstly, let simplify the Melnikov function $ M_{0}(h)$ given by \eqref{mel:M0} for the case $\tau_{\scriptscriptstyle LS}\ne\tau_{\scriptscriptstyle RS}$. For this propose, consider the orbit $(x_{\scriptscriptstyle R}(t),y_{\scriptscriptstyle R}(t))$ of the system $\eqref{eq:01}|_{\epsilon=0}$, such that $(x_{\scriptscriptstyle R}(0),y_{\scriptscriptstyle R}(0))=(1,h)$,  given by
\begin{equation*}
	\begin{aligned}
		x_{\scriptscriptstyle R}(t)=\,&-\frac{e^{-t \omega_{\scriptscriptstyle RS}}}{2\omega_{\scriptscriptstyle RS}}\Big(b_{\scriptscriptstyle R} h-b_{\scriptscriptstyle R}e^{2t \omega_{\scriptscriptstyle RS}}h-2e^{t \omega_{\scriptscriptstyle RS}}\omega_{\scriptscriptstyle RS}+b_{\scriptscriptstyle R}\tau_{\scriptscriptstyle RS}-2b_{\scriptscriptstyle R}e^{t \omega_{\scriptscriptstyle RS}}\tau_{\scriptscriptstyle RS} +b_{\scriptscriptstyle R}e^{2t \omega_{\scriptscriptstyle RS}}\tau_{\scriptscriptstyle RS}\Big), \\
		y_{\scriptscriptstyle R}(t)=\,&-\frac{e^{-t \omega_{\scriptscriptstyle RS}}}{2\omega_{\scriptscriptstyle RS}}\Big(-a_{\scriptscriptstyle R}h+a_{\scriptscriptstyle R}e^{2t \omega_{\scriptscriptstyle RS}}h-\omega_{\scriptscriptstyle RS}h- e^{2t \omega_{\scriptscriptstyle RS}}\omega_{\scriptscriptstyle RS}h-a_{\scriptscriptstyle R}\tau_{\scriptscriptstyle RS} +2a_{\scriptscriptstyle R}e^{t \omega_{\scriptscriptstyle RS}}\tau_{\scriptscriptstyle RS}\\
		&-a_{\scriptscriptstyle R}e^{2t \omega_{\scriptscriptstyle RS}}\tau_{\scriptscriptstyle RS}-\omega_{\scriptscriptstyle RS}\tau_{\scriptscriptstyle RS}+e^{2t \omega_{\scriptscriptstyle RS}}\omega_{\scriptscriptstyle RS}\tau_{\scriptscriptstyle RS}\Big).
	\end{aligned}
\end{equation*}
The flight time of the orbit $(x_{\scriptscriptstyle R}(t),y_{\scriptscriptstyle R}(t))$, from $A(h)=(1,h)$ to $A_1(h)=(1,-h)$, is 
$$t_{\scriptscriptstyle R}=\frac{1}{\omega_{\scriptscriptstyle RS}}\log\bigg(\frac{h+\tau_{\scriptscriptstyle RS}}{\tau_{\scriptscriptstyle RS}-h}\bigg).$$
Now, for $g_{\scriptscriptstyle R}$ and $f_{\scriptscriptstyle R}$ defined in \eqref{eq:02} and \eqref{eq:03}, respectively, we have
\begin{equation}\label{sys:rsss0}
	\begin{aligned}
		I_{\scriptscriptstyle R}^0&=\int_{\widehat{AA_1}}g_{\scriptscriptstyle R}dx-f_{\scriptscriptstyle R}dy \\
		& =\int_{0}^{t_{\scriptscriptstyle R}}\Big(g_{\scriptscriptstyle R}(x_{\scriptscriptstyle R}(t), y_{\scriptscriptstyle R}(t)) x'_{\scriptscriptstyle R}(t) - 
		f_{\scriptscriptstyle R}(x_{\scriptscriptstyle R}(t), y_{\scriptscriptstyle R}(t)) y'_{\scriptscriptstyle R}(t)\Big)dt \\
		& =\int_{0}^{t_{\scriptscriptstyle R}}\bigg(\sum_{i+j=0}^{1}q_{ij}x^i_{\scriptscriptstyle R}(t)y^j_{\scriptscriptstyle R}(t)x'_{\scriptscriptstyle R}(t) - \sum_{i+j=0}^{1}p_{ij}x^i_{\scriptscriptstyle R}(t)y^j_{\scriptscriptstyle R}(t)y'_{\scriptscriptstyle R}(t)\bigg)dt\\
		& = \alpha_1f_0(h)+\alpha_2f_{\scriptscriptstyle RS}(h),		
	\end{aligned}
\end{equation}
with 
$$
\alpha_1 =\frac{1}{\omega_{\scriptscriptstyle RS}}\Big(2 (p_{00} + p_{10}) \omega_{\scriptscriptstyle RS} + b_{\scriptscriptstyle R} (p_{10} + q_{01}) \tau_{\scriptscriptstyle RS}\Big)\quad\text{and}\quad
\alpha_2  = \frac{b_{\scriptscriptstyle R}}{2\omega_{\scriptscriptstyle RS}}\Big(p_{10} + q_{01}\Big).		
$$
The orbit $(x_{\scriptscriptstyle C1}(t),y_{\scriptscriptstyle C1}(t))$ of the system $\eqref{eq:01}|_{\epsilon=0}$, such that $(x_{\scriptscriptstyle C1}(0),y_{\scriptscriptstyle C1}(0))=(1,-h)$, is given by
\begin{equation*}
	\begin{aligned}
		x_{\scriptscriptstyle C1}(t)&=\frac{e^{-t}}{2}(h+1)-\frac{e^{t}}{2}(h-1), \\
		y_{\scriptscriptstyle C1}(t)&=-\frac{e^{-t}}{2}(h+1)-\frac{e^{t}}{2}(h-1).
	\end{aligned}
\end{equation*}
The flight time of the orbit $(x_{\scriptscriptstyle C1}(t),y_{\scriptscriptstyle C1}(t))$, from $A_1(h)=(1,-h)$ to $A_2(h)=(-1,-h)$, is
$$t_{\scriptscriptstyle C1}=\log\bigg(\frac{h+1}{h-1}\bigg).$$
Now, for $g_{\scriptscriptstyle C}$ and $f_{\scriptscriptstyle C}$ defined in \eqref{eq:02} and \eqref{eq:03}, respectively, we obtain 
\begin{equation}\label{sys:c1sss0}
	\begin{aligned}
		\bar I_{\scriptscriptstyle C}^0 = & \int_{\widehat{A_1A_2}}g_{\scriptscriptstyle C}dx-f_{\scriptscriptstyle C}dy\\
		=& \int_{0}^{t_{\scriptscriptstyle C1}}\hspace{-0.1cm}\Big(g_{\scriptscriptstyle C}(x_{\scriptscriptstyle C1}(t), y_{\scriptscriptstyle C1}(t)) x'_{\scriptscriptstyle C1}(t)  - 
		f_{\scriptscriptstyle C}(x_{\scriptscriptstyle C1}(t), y_{\scriptscriptstyle C1}(t)) y'_{\scriptscriptstyle C1}(t)\Big)dt,\\
		= & \int_{0}^{t_{\scriptscriptstyle C1}}\bigg(\sum_{i+j=0}^{1}v_{ij}x^i_{\scriptscriptstyle C1}(t)y^j_{\scriptscriptstyle C1}(t)x'_{\scriptscriptstyle C1}(t) - \sum_{i+j=0}^{1}u_{ij}x^i_{\scriptscriptstyle C1}(t)y^j_{\scriptscriptstyle C1}(t)y'_{\scriptscriptstyle C1}(t)\bigg)dt\\
		=& -2v_{00}+\alpha_3f_0(h)+\alpha_4f_{\scriptscriptstyle C}^0(h),		
	\end{aligned}
\end{equation}
with
$$
\alpha_3 = v_{01}-u_{10} \quad\text{and}\quad
\alpha_4 = \frac{1}{2}\Big(u_{10} + v_{01}\Big).		
$$
The orbit $(x_{\scriptscriptstyle L}(t),y_{\scriptscriptstyle L}(t))$ of the system $\eqref{eq:01}|_{\epsilon=0}$, such that $(x_{\scriptscriptstyle L}(0),y_{\scriptscriptstyle L}(0))=(-1,-h)$, is given by
\begin{equation*}
	\begin{aligned}
		x_{\scriptscriptstyle L}(t)=\,&\frac{e^{-t \omega_{\scriptscriptstyle LS}}}{2\omega_{\scriptscriptstyle LS}}\Big(b_{\scriptscriptstyle L} h-b_{\scriptscriptstyle L}e^{2t \omega_{\scriptscriptstyle LS}}h-2e^{t \omega_{\scriptscriptstyle LS}}\omega_{\scriptscriptstyle LS}+b_{\scriptscriptstyle L}\tau_{\scriptscriptstyle LS}-2b_{\scriptscriptstyle L}e^{t \omega_{\scriptscriptstyle LS}}\tau_{\scriptscriptstyle LS} +b_{\scriptscriptstyle L}e^{2t \omega_{\scriptscriptstyle LS}}\tau_{\scriptscriptstyle LS}\Big), \\
		y_{\scriptscriptstyle L}(t)=\,&\frac{e^{-t \omega_{\scriptscriptstyle LS}}}{2\omega_{\scriptscriptstyle LS}}\Big(-a_{\scriptscriptstyle L}h+a_{\scriptscriptstyle L}e^{2t \omega_{\scriptscriptstyle LS}}h-\omega_{\scriptscriptstyle LS}h- e^{2t \omega_{\scriptscriptstyle LS}}\omega_{\scriptscriptstyle LS}h-a_{\scriptscriptstyle L}\tau_{\scriptscriptstyle LS}+2a_{\scriptscriptstyle L}e^{t \omega_{\scriptscriptstyle LS}}\tau_{\scriptscriptstyle LS}\\
		& -a_{\scriptscriptstyle L}e^{2t \omega_{\scriptscriptstyle LS}}\tau_{\scriptscriptstyle LS}-\omega_{\scriptscriptstyle LS}\tau_{\scriptscriptstyle LS}+e^{2t \omega_{\scriptscriptstyle LS}}\omega_{\scriptscriptstyle LS}\tau_{\scriptscriptstyle LS}\Big).
	\end{aligned}
\end{equation*}
The flight time of the orbit $(x_{\scriptscriptstyle L}(t),y_{\scriptscriptstyle L}(t))$, from $A_2(h)=(-1,-h)$ to $A_3(h)=(-1,h)$, is 
$$t_{\scriptscriptstyle L}=\frac{1}{\omega_{\scriptscriptstyle LS}}\log\bigg(\frac{h+\tau_{\scriptscriptstyle LS}}{\tau_{\scriptscriptstyle LS}-h}\bigg).$$
Now, for $g_{\scriptscriptstyle L}$ and $f_{\scriptscriptstyle L}$ defined in \eqref{eq:02} and \eqref{eq:03}, respectively, we have
\begin{equation}\label{sys:lsss0}
	\begin{aligned}
		I_{\scriptscriptstyle L}^0 &=\int_{\widehat{A_2A_3}}g_{\scriptscriptstyle L}dx-f_{\scriptscriptstyle L}dy \\
		& =\int_{0}^{t_{\scriptscriptstyle L}}\Big(g_{\scriptscriptstyle L}(x_{\scriptscriptstyle L}(t), y_{\scriptscriptstyle L}(t)) x'_{\scriptscriptstyle L}(t) - 
		f_{\scriptscriptstyle L}(x_{\scriptscriptstyle L}(t), y_{\scriptscriptstyle L}(t)) y'_{\scriptscriptstyle L}(t)\Big)dt \\
		& = \int_{0}^{t_{\scriptscriptstyle L}}\bigg(\sum_{i+j=0}^{1}s_{ij}x^i_{\scriptscriptstyle L}(t)y^j_{\scriptscriptstyle L}(t)x'_{\scriptscriptstyle L}(t) - \sum_{i+j=0}^{1}r_{ij}x^i_{\scriptscriptstyle L}(t)y^j_{\scriptscriptstyle L}(t)y'_{\scriptscriptstyle L}(t)\bigg)dt\\
		& = \alpha_5f_0(h)+\alpha_6f_{\scriptscriptstyle LS}(h),		
	\end{aligned}
\end{equation}
with 
$$
\alpha_5 =\frac{1}{\omega_{\scriptscriptstyle LS}}\Big(2 (r_{10}-r_{00}) \omega_{\scriptscriptstyle LS} + b_{\scriptscriptstyle L} (r_{10} + s_{01}) \tau_{\scriptscriptstyle LS}\Big)\quad\text{and}\quad
\alpha_6  = \frac{b_{\scriptscriptstyle L}}{2\omega_{\scriptscriptstyle LS}}\Big(r_{10} + s_{01}\Big).		
$$
The orbit $(x_{\scriptscriptstyle C2}(t),y_{\scriptscriptstyle C2}(t))$ of the system $\eqref{eq:01}|_{\epsilon=0}$, such that $(x_{\scriptscriptstyle C2}(0),$ $y_{\scriptscriptstyle C2}(0))=(-1,h)$, is given by
\begin{equation*}
	\begin{aligned}
		x_{\scriptscriptstyle C2}(t)&=-\frac{e^{-t}}{2}(h+1)+\frac{e^{t}}{2}(h-1), \\
		y_{\scriptscriptstyle C2}(t)&=\frac{e^{-t}}{2}(h+1)+\frac{e^{t}}{2}(h-1).
	\end{aligned}
\end{equation*}
The flight time of the orbit $(x_{\scriptscriptstyle C2}(t),y_{\scriptscriptstyle C2}(t))$, from $A_3(h)=(-1,h)$ to $A(h)=(1,h)$, is
$$t_{\scriptscriptstyle C2}=\log\bigg(\frac{h+1}{h-1}\bigg).$$
Now, for $g_{\scriptscriptstyle C}$ and $f_{\scriptscriptstyle C}$ defined in \eqref{eq:02} and \eqref{eq:03}, respectively, we obtain 
\begin{equation}\label{sys:c2sss0}
	\begin{aligned}
		I_{\scriptscriptstyle C}^0=&\int_{\widehat{A_3A}}g_{\scriptscriptstyle C}dx-f_{\scriptscriptstyle C}dy \\
		=&\int_{0}^{t_{\scriptscriptstyle C2}}\Big(g_{\scriptscriptstyle C}(x_{\scriptscriptstyle C2}(t), y_{\scriptscriptstyle C2}(t))   x'_{\scriptscriptstyle C2}(t)  - 
		f_{\scriptscriptstyle C}(x_{\scriptscriptstyle C2}(t), y_{\scriptscriptstyle C2}(t)) y'_{\scriptscriptstyle C2}(t)\Big)dt \\
		= & \int_{0}^{t_{\scriptscriptstyle C2}}\bigg(\sum_{i+j=0}^{1}v_{ij}x^i_{\scriptscriptstyle C2}(t)y^j_{\scriptscriptstyle C2}(t)x'_{\scriptscriptstyle C2}(t) - \sum_{i+j=0}^{1}u_{ij}x^i_{\scriptscriptstyle C2}(t)y^j_{\scriptscriptstyle C2}(t)y'_{\scriptscriptstyle C2}(t)\bigg)dt\\
		= & \,\,2v_{00}+\alpha_3f_0(h)+\alpha_4f_{\scriptscriptstyle C}^0(h).		
	\end{aligned}
\end{equation}
Hence, by Corollary \ref{col:mel}, the first order Melnivov function $M_0(h)$ associated with system \eqref{eq:01} is given by
\begin{equation}\label{eq:mel0sss}
	M_0(h) = b_{\scriptscriptstyle R} I^{0}_{\scriptscriptstyle C}+\frac{b_{\scriptscriptstyle R}}{b_{\scriptscriptstyle L}}I^{0}_{\scriptscriptstyle L}+b_{\scriptscriptstyle R} \bar I^{0}_{\scriptscriptstyle C}+I^{0}_{\scriptscriptstyle R}. 
\end{equation}
Replacing \eqref{sys:rsss0}, \eqref{sys:c1sss0}, \eqref{sys:lsss0} and \eqref{sys:c2sss0} in \eqref{eq:mel0sss} we obtain \eqref{mel:sss00}, with
$$k_0^0=\alpha_1 + 2 b_{\scriptscriptstyle R}\alpha_3 + \frac{b_{\scriptscriptstyle R}}{b_{\scriptscriptstyle L}} \alpha_5,\quad k_0^1=2 b_{\scriptscriptstyle R} \alpha_4,\quad k_0^2=\alpha_2\quad\text{and}\quad k_0^3=\frac{b_{\scriptscriptstyle R}}{b_{\scriptscriptstyle L}}\alpha_6.$$
Now, let simplify the Melnikov function $ M_{1}(h)$ given by \eqref{mel:M1} for the case  $\tau_{\scriptscriptstyle LS}\ne\tau_{\scriptscriptstyle RS}$. Similarly as in \eqref{sys:rsss0}, we have that
\begin{equation}\label{sys:rsss1}
	\begin{aligned}
		I_{\scriptscriptstyle R}^1&=\int_{\widehat{BB_1}}g_{\scriptscriptstyle R}dx-f_{\scriptscriptstyle R}dy \\
		& =\int_{0}^{t_{\scriptscriptstyle R}}\Big(g_{\scriptscriptstyle R}(x_{\scriptscriptstyle R}(t), y_{\scriptscriptstyle R}(t)) x'_{\scriptscriptstyle R}(t) - 
		f_{\scriptscriptstyle R}(x_{\scriptscriptstyle R}(t), y_{\scriptscriptstyle R}(t)) y'_{\scriptscriptstyle R}(t)\Big)dt \\
		& =\int_{0}^{t_{\scriptscriptstyle R}}\bigg(\sum_{i+j=0}^{1}q_{ij}x^i_{\scriptscriptstyle R}(t)y^j_{\scriptscriptstyle R}(t)x'_{\scriptscriptstyle R}(t) - \sum_{i+j=0}^{1}p_{ij}x^i_{\scriptscriptstyle R}(t)y^j_{\scriptscriptstyle R}(t)y'_{\scriptscriptstyle R}(t)\bigg)dt\\
		& = \alpha_1f_0(h)+\alpha_2f_{\scriptscriptstyle RS}(h).	
	\end{aligned}
\end{equation}
The orbit $(x_{\scriptscriptstyle C}(t),y_{\scriptscriptstyle C}(t))$ of the system $\eqref{eq:01}|_{\epsilon=0}$, such that $(x_{\scriptscriptstyle C}(0),y_{\scriptscriptstyle C}(0))=(1,-h)$, is given by
\begin{equation*}
	\begin{aligned}
		x_{\scriptscriptstyle C}(t)&=\frac{e^{-t}}{2}(h+1)-\frac{e^{t}}{2}(h-1), \\
		y_{\scriptscriptstyle C}(t)&=-\frac{e^{-t}}{2}(h+1)-\frac{e^{t}}{2}(h-1).
	\end{aligned}
\end{equation*}
The flight time of the orbit $(x_{\scriptscriptstyle C}(t),y_{\scriptscriptstyle C}(t))$, from $B_1(h)=(1,-h)$ to $B(h)=(1,h)$, is
$$t_{\scriptscriptstyle C}=\log\bigg(\frac{h+1}{1-h}\bigg).$$
Now, for $g_{\scriptscriptstyle C}$ and $f_{\scriptscriptstyle C}$ defined in \eqref{eq:02} and \eqref{eq:03}, respectively, we obtain 
\begin{equation}\label{sys:c1sss1}
	\begin{aligned}
		\bar I_{\scriptscriptstyle C}^1 = & \int_{\widehat{B_1B}}g_{\scriptscriptstyle C}dx-f_{\scriptscriptstyle C}dy\\
		=& \int_{0}^{t_{\scriptscriptstyle C}}\hspace{-0.1cm}(g_{\scriptscriptstyle C}(x_{\scriptscriptstyle C}(t), y_{\scriptscriptstyle C}(t)) x'_{\scriptscriptstyle C}(t)  - 
		f_{\scriptscriptstyle C}(x_{\scriptscriptstyle C}(t), y_{\scriptscriptstyle C}(t)) y'_{\scriptscriptstyle C}(t))dt,\\
		= & \int_{0}^{t_{\scriptscriptstyle C}}\bigg(\sum_{i+j=0}^{1}v_{ij}x^i_{\scriptscriptstyle C}(t)y^j_{\scriptscriptstyle C}(t)x'_{\scriptscriptstyle C}(t) - \sum_{i+j=0}^{1}u_{ij}x^i_{\scriptscriptstyle C}(t)y^j_{\scriptscriptstyle C}(t)y'_{\scriptscriptstyle C}(t)\bigg)dt\\
		=& \,\,\alpha_7f_0(h)+\alpha_4f_{\scriptscriptstyle C}(h),		
	\end{aligned}
\end{equation}
with 
$$\alpha_7=-2 u_{00} - u_{10} + v_{01}.$$ 
Hence, by Corollary \ref{col:mel}, the first order Melnivov function $M_1(h)$ associated with system \eqref{eq:01} is given by
\begin{equation}\label{eq:mel1sss}
	M_1(h) = b_{\scriptscriptstyle R} I^{1}_{\scriptscriptstyle C}+I^{1}_{\scriptscriptstyle R}. 
\end{equation}
Replacing \eqref{sys:rsss1} and \eqref{sys:c1sss1} in \eqref{eq:mel1sss} we obtain \eqref{mel:sss10}, with
$$k_1^0=\alpha_1 +b_{\scriptscriptstyle R}\alpha_7,\quad k_1^1= b_{\scriptscriptstyle R} \alpha_4,\quad\text{and}\quad k_1^2=\alpha_2.$$
Now, let simplify the Melnikov function $ M_{2}(h)$ given by \eqref{mel:M2} for the case  $\tau_{\scriptscriptstyle LS}\ne\tau_{\scriptscriptstyle RS}$. Similarly as in \eqref{sys:lsss0}, we have that
\begin{equation}\label{sys:lsss22}
	\begin{aligned}
		I_{\scriptscriptstyle L}^2 &=\int_{\widehat{C_1C}}g_{\scriptscriptstyle L}dx-f_{\scriptscriptstyle L}dy \\
		& =\int_{0}^{t_{\scriptscriptstyle L}}\Big(g_{\scriptscriptstyle L}(x_{\scriptscriptstyle L}(t), y_{\scriptscriptstyle L}(t)) x'_{\scriptscriptstyle L}(t) - 
		f_{\scriptscriptstyle L}(x_{\scriptscriptstyle L}(t), y_{\scriptscriptstyle L}(t)) y'_{\scriptscriptstyle L}(t)\Big)dt \\
		& = \int_{0}^{t_{\scriptscriptstyle L}}\bigg(\sum_{i+j=0}^{1}s_{ij}x^i_{\scriptscriptstyle L}(t)y^j_{\scriptscriptstyle L}(t)x'_{\scriptscriptstyle L}(t) - \sum_{i+j=0}^{1}r_{ij}x^i_{\scriptscriptstyle L}(t)y^j_{\scriptscriptstyle L}(t)y'_{\scriptscriptstyle L}(t)\bigg)dt\\
		& = \alpha_5f_0(h)+\alpha_6f_{\scriptscriptstyle LS}(h),		
	\end{aligned}
\end{equation}
The orbit $(x_{\scriptscriptstyle C}(t),y_{\scriptscriptstyle C}(t))$ of the system $\eqref{eq:01}|_{\epsilon=0}$, such that $(x_{\scriptscriptstyle C}(0),y_{\scriptscriptstyle C}(0))=(-1,h)$, is given by
\begin{equation*}
	\begin{aligned}
		x_{\scriptscriptstyle C}(t)&=-\frac{e^{-t}}{2}(h+1)+\frac{e^{t}}{2}(h-1), \\
		y_{\scriptscriptstyle C}(t)&=\frac{e^{-t}}{2}(h+1)+\frac{e^{t}}{2}(h-1).
	\end{aligned}
\end{equation*}
The flight time of the orbit $(x_{\scriptscriptstyle C}(t),y_{\scriptscriptstyle C}(t))$, from $C(h)=(-1,h)$ to $C_1(h)=(-1,-h)$, is
$$t_{\scriptscriptstyle C}=\log\bigg(\frac{h+1}{1-h}\bigg).$$
Now, for $g_{\scriptscriptstyle C}$ and $f_{\scriptscriptstyle C}$ defined in \eqref{eq:02} and \eqref{eq:03}, respectively, we obtain 
\begin{equation}\label{sys:c2sss2}
	\begin{aligned}
		\bar I_{\scriptscriptstyle C}^2 = & \int_{\widehat{CC1}}g_{\scriptscriptstyle C}dx-f_{\scriptscriptstyle C}dy\\
		=& \int_{0}^{t_{\scriptscriptstyle C}}\hspace{-0.1cm}\Big(g_{\scriptscriptstyle C}(x_{\scriptscriptstyle C}(t), y_{\scriptscriptstyle C}(t)) x'_{\scriptscriptstyle C}(t)  - 
		f_{\scriptscriptstyle C}(x_{\scriptscriptstyle C}(t), y_{\scriptscriptstyle C}(t)) y'_{\scriptscriptstyle C}(t)\Big)dt,\\
		= & \int_{0}^{t_{\scriptscriptstyle C}}\bigg(\sum_{i+j=0}^{1}v_{ij}x^i_{\scriptscriptstyle C}(t)y^j_{\scriptscriptstyle C}(t)x'_{\scriptscriptstyle C}(t) - \sum_{i+j=0}^{1}u_{ij}x^i_{\scriptscriptstyle C}(t)y^j_{\scriptscriptstyle C}(t)y'_{\scriptscriptstyle C}(t)\bigg)dt\\
		=& \,\,\alpha_8f_0(h)+\alpha_4f_{\scriptscriptstyle C}(h),		
	\end{aligned}
\end{equation}
with 
$$\alpha_8=2 u_{00} - u_{10} + v_{01}.$$ 
Hence, by Corollary \ref{col:mel}, the first order Melnivov function $M_2(h)$ associated with system \eqref{eq:01} is given by
\begin{equation}\label{eq:mel2sss}
	M_2(h) = \frac{1}{b_{\scriptscriptstyle L}} I^{2}_{\scriptscriptstyle L}+I^{2}_{\scriptscriptstyle C}. 
\end{equation}
Replacing \eqref{sys:lsss22} and \eqref{sys:c2sss2} in \eqref{eq:mel2sss} we obtain \eqref{mel:sss20}, with
$$k_2^0=\frac{1}{b_{\scriptscriptstyle L}}\alpha_5 + \alpha_8,\quad k_2^1= \alpha_4,\quad\text{and}\quad k_2^2=\frac{1}{b_{\scriptscriptstyle L}}\alpha_6.$$
Finally, replacing $\tau_{\scriptscriptstyle LS}=\tau_{\scriptscriptstyle RS}$ in \eqref{mel:sss00}, \eqref{mel:sss10} and \eqref{mel:sss20} we obtain the expression \eqref{mel:sss01}, \eqref{mel:sss11} and \eqref{mel:sss21}, with
$$\tilde{k}_0^0=\alpha_1 + 2 b_{\scriptscriptstyle R}\alpha_3 + \frac{b_{\scriptscriptstyle R}}{b_{\scriptscriptstyle L}\omega_{\scriptscriptstyle LS}} \Big(2 (r_{10}-r_{00}) \omega_{\scriptscriptstyle LS} + b_{\scriptscriptstyle L}  (r_{10} + s_{01}) \tau_{\scriptscriptstyle RS}\Big),\,\,\, \tilde{k}_0^1=2 b_{\scriptscriptstyle R} \alpha_4,$$
$$\tilde{k}_0^2=\alpha_2+\frac{b_{\scriptscriptstyle R}}{b_{\scriptscriptstyle L}}\alpha_6,\quad\text{and}\quad\tilde{k}_i^j=k_i^j|_{\tau_{\scriptscriptstyle LS}=\tau_{\scriptscriptstyle RS}}, \quad\text{for}\,\,i=1,2,\,\, j=0,1,2.$$
\end{proof}

\bigskip

The following two theorems provide expressions for the first order Melnikov functions associated with system \eqref{eq:01} when the system $\eqref{eq:01}|_{\epsilon=0}$ is of the types CSS and CSC. The proof of these results is similarity to proof of Theorem \ref{theo:sss}, and will be omitted to simplify the text, but we emphasize the follow remark.

\begin{remark}
	In the cases CSS and CSC, the flight time of the orbits $(x_{\scriptscriptstyle R}(t),y_{\scriptscriptstyle R}(t))$ and $(x_{\scriptscriptstyle L}(t),y_{\scriptscriptstyle L}(t))$ of the system $\eqref{eq:01}|_{\epsilon=0}$ such that $(x_{\scriptscriptstyle R}(0),y_{\scriptscriptstyle L}(0))=(1,h)$ and $(x_{\scriptscriptstyle L}(0),y_{\scriptscriptstyle L}(0))=(-1,-h)$, are given by
	\begin{equation*}
		\begin{aligned}
			t_{\scriptscriptstyle R} =&\, \frac{1}{\omega_{\scriptscriptstyle RC}}\bigg(2\pi \mu_1+(-1)^{\mu_1}\arccos\bigg(\frac{\tau_{\scriptscriptstyle RC}^2-h^2}{\tau_{\scriptscriptstyle RC}^2+h^2}\bigg)\bigg),\\
			t_{\scriptscriptstyle L} =&\, \frac{1}{\omega_{\scriptscriptstyle LC}}\bigg(2\pi \mu_2+(-1)^{\mu_2}\arccos\bigg(\frac{\tau_{\scriptscriptstyle LC}^2-h^2}{\tau_{\scriptscriptstyle LC}^2+h^2}\bigg)\bigg), 
		\end{aligned}
	\end{equation*}	
	respectively, with $\mu_i=0$ (resp. $\mu_i=1$), $i=1,2$, if the right or left subsystem from $\eqref{eq:01}|_{\epsilon=0}$ has a virtual (resp. real) center.
\end{remark}

\begin{theorem}\label{theo:css}
	Suppose that system $\eqref{eq:01}|_{\epsilon=0}$ is of the type CSS. Then the first order Melnikov functions given by  \eqref{mel:M0}, \eqref{mel:M1} and \eqref{mel:M2} associated with system \eqref{eq:01} can be expressed as	
	\begin{eqnarray*}
		&& M_0(h)=k_0^0f_0(h)+k_0^1f_{\scriptscriptstyle C}^{0}(h)+k_0^2f_{\scriptscriptstyle RS}(h)+k_0^3f_{\scriptscriptstyle LC}(h),\quad\quad h\in(1,\tau_{\scriptscriptstyle RS}),\label{mel:css00}\\
		&& M_1(h)=k_1^0f_0(h)+k_1^1f_{\scriptscriptstyle C}(h)+k_1^2f_{\scriptscriptstyle RS}(h),\quad\quad\quad\quad\quad\quad\quad h\in(0,1),\label{mel:css01}\\
		&& M_2(h)=k_2^0f_0(h)+k_2^1f_{\scriptscriptstyle C}(h)+k_2^2f_{\scriptscriptstyle LC}(h),\quad\quad\quad\quad\quad\quad\quad h\in(0,1),\label{mel:css02}
	\end{eqnarray*}
 where the functions $f_0(h)$, $f_{\scriptscriptstyle C}^0(h)$, $f_{\scriptscriptstyle C}(h)$, $f_{\scriptscriptstyle RS}(h)$ and $f_{\scriptscriptstyle LC}(h)$ are the ones defined in \eqref{eq:func}, with $\mu_2=0$ (resp. $\mu_2=1$) if the left subsystem from $\eqref{eq:01}|_{\epsilon=0}$ has a virtual (resp. real) center. Here the coefficients $k_i^j$, for $i=0,1,2$ and $j=0,1,2,3$, depend on the parameters of system \eqref{eq:01}.
\end{theorem}

\medskip

\begin{theorem}\label{theo:csc}
	Suppose that system $\eqref{eq:01}|_{\epsilon=0}$ is of the type CSC. Then the first order Melnikov functions given by  \eqref{mel:M0}, \eqref{mel:M1} and \eqref{mel:M2} associated with system \eqref{eq:01} can be expressed as	
	\begin{eqnarray*}
		&& M_0(h)=k_0^0f_0(h)+k_0^1f_{\scriptscriptstyle C}^{0}(h)+k_0^2f_{\scriptscriptstyle RC}(h)+k_0^3f_{\scriptscriptstyle LC}(h),\quad\quad h\in(1,\infty),\label{mel:csc00}\\
		&& M_1(h)=k_1^0f_0(h)+k_1^1f_{\scriptscriptstyle C}(h)+k_1^2f_{\scriptscriptstyle RC}(h),\quad\quad\quad\quad\quad\quad\quad h\in(0,1),\label{mel:csc01}\\
		&& M_2(h)=k_2^0f_0(h)+k_2^1f_{\scriptscriptstyle C}(h)+k_2^2f_{\scriptscriptstyle LC}(h),\quad\quad\quad\quad\quad\quad\quad h\in(0,1)\label{mel:csc02},
	\end{eqnarray*}
	where the functions $f_0(h)$, $f_{\scriptscriptstyle C}^0(h)$, $f_{\scriptscriptstyle C}(h)$, $f_{\scriptscriptstyle RC}(h)$ and $f_{\scriptscriptstyle LC}(h)$ are the ones defined in \eqref{eq:func}, with $\mu_i=0$ (resp. $\mu_i=1$), $i=1,2$, if the right or left subsystem from $\eqref{eq:01}|_{\epsilon=0}$ has a virtual (resp. real) center. Here the coefficients $k_i^j$, for $i=0,1,2$ and $j=0,1,2,3$, depend on the parameters of system \eqref{eq:01}.
\end{theorem}

\bigskip


Before proving the Theorem \ref{the:01}, we will need the following results.

\medskip

Consider the functions $F,G,H:\mathbb{R}\rightarrow\mathbb{R}$ given by 
\begin{eqnarray}
	&&F(h)=\sum_{j=0}^{2}C_j^{F}(\delta) (h-\tau)^j+\mathcal{O}((h-\tau)^{3}),\label{func:F}\\
	&&G(h)=\sum_{j=0}^{2}C_j^{G}(\delta) (h-\tau)^j+\mathcal{O}((h-\tau)^{3}),\label{func:G}\\
	&&H(h)=\sum_{j=0}^{2}C_j^{H}(\delta) (h-\tau)^j+\mathcal{O}((h-\tau)^{3}),\label{func:H}
\end{eqnarray}
such that $C_j^{F}=C_j^{G}+C_j^{H}$,  $\tau\geq0$ and the coefficients  $C_j^{i}(\delta)$, $i=F,G,H$, $j=0,\dots,n$, depending on the parameters $\delta=(\delta_1,\dots,\delta_m)\in\mathbb{R}^m$. For each fixed $\delta$, denote by $N_F(\delta)$, $N_G(\delta)$ and $N_H(\delta)$ the number of zeros of the function $F$, $G$ and $H$, respectively. Then we have the follow proposition.

\begin{proposition}\label{the:coef}  Suppose that there exist   $\tilde{\delta}\in\mathbb{R}^{m}$, with $m\geq 5$, such that 
	\begin{equation*}\label{cond:1}
	C_{j}^G(\tilde{\delta})=C_{j}^{H}(\tilde{\delta})=0,\quad j=0,1,2,
	\end{equation*}
	\begin{equation}\label{cond:rank}
		\textnormal{rank}\,\frac{\partial(C_{0}^G,C_{1}^G,C_{2}^G,C_{0}^H, C_{1}^H)}{\partial(\delta_1,\dots,\delta_{5},\delta_{6},\dots,\delta_m)}\tilde{(\delta)}=5,
	\end{equation}
and
\begin{equation}\label{cond:coef}
 C_2^{H}=\alpha_{0}^G C_{0}^G+\alpha_{1}^G C_{1}^G+\alpha_{2}^G C_{2}^G+\alpha_{0}^H C_{0}^H+\alpha_{1}^H C_{1}^H,
\end{equation}
where $\alpha_i^k$, $i=0,1$ and $k=G,H$ are constants.
	If $\alpha_{2}^G<0$ and $\alpha_{2}^G\ne-1$ (resp. $\alpha_{2}^G\geq0$ or $\alpha_{2}^G=-1$), then the $\max \{N_F(\delta)+N_G(\delta)+N_H(\delta)\}$ is at least six (resp. five) in a neighborhood of $h=\tau$ for all $\delta$ near $\tilde{\delta}$. Moreover, the following configurations of zeros $(N_F,N_G,N_H)$, for $N_F,N_G,N_H\in\{0,1,2\}$, are achievable. In particular, when $\alpha_{2}^G<0$ and $\alpha_{2}^G\ne-1$ (resp. $\alpha_{2}^G\geq0$ or $\alpha_{2}^G=-1$), we have $(2,2,2)$ (resp. $(2,2,1)$ or $(1,2,2)$).
		
\end{proposition}
\begin{proof}
	By the condition \eqref{cond:rank} we can assume that
	$$\det\frac{\partial(C_{0}^G,C_{1}^G,C_{2}^G,C_{0}^H,C_{1}^H)}{\partial(\delta_1,\delta_{2},\delta_{3},\delta_{4},\delta_{5})}\tilde{(\delta)}\ne0,$$
	Then the change of parameters $\tilde{C}_j^G=C_j^G(\delta_1,\dots,\delta_{5},\tilde{\delta}_{6},\dots,\tilde{\delta}_{m})$, $j=0,1,2$, and $\tilde{C}_j^H=C_j^H(\delta_1,\dots,\delta_{5},\tilde{\delta}_{6},\dots,\tilde{\delta}_{m})$, $j=0,1$, has inverse $\delta_j(\tilde{C}_{0}^G,\tilde{C}_{1}^G,\tilde{C}_{2}^G,\tilde{C}_{0}^H,\tilde{C}_{1}^H)$, $j=1,\dots,5$, and can write \eqref{func:F}, \eqref{func:G} and \eqref{func:H} as
	\begin{eqnarray*}
		&&F(h)=\tilde{C}_0^F+\tilde{C}_1^F(h-\tau)+\tilde{C}_{2}^F(h-\tau)^{2}+\mathcal{O}((h-\tau)^{3}),\\
		&&G(h)=\tilde{C}_0^G+\tilde{C}_1^G(h-\tau)+\tilde{C}_{2}^G(h-\tau)^{2}+\mathcal{O}((h-\tau)^{3}),\\
		&&H(h)=\tilde{C}_0^H+\tilde{C}_1^H(h-\tau)+\tilde{C}_{2}^H(h-\tau)^{2}+\mathcal{O}((h-\tau)^{3}),
	\end{eqnarray*}
	with $\tilde{C}_{i}^G(\tilde{\delta})=0$, $i=0,1,2$, and  $\tilde{C}_{j}^{H}(\tilde{\delta})=0$,  $j=0,1$. 
	
	Let $\tilde{C}_{2}^G>0$ and $h_2^{G}$ such that $0<\tau-h_2^G\ll 1$ and
	$$\tilde{C}_{2}^G(h_{2}^G-\tau)^{2}>0.$$ 
	For $\tilde{C}_{2}^G$ fixed and $|\tilde{C}_{j}^k|\ll \tilde{C}_{2}^G$, $j=0,1,2$ and $k=G,H$, by equation $\eqref{cond:coef}$, we have $\tilde{C}_{2}^H=\alpha_{2}^G \tilde{C}_{2}^G+\mathcal{O}_1\big(\tilde{C}_{0}^G,\tilde{C}_{1}^G,\tilde{C}_{0}^H,\tilde{C}_{1}^H\big)$. Thus, assuming these conditions, we proceed with the proof.
	
	 As $\tilde{C}_{2}^F=\tilde{C}_{2}^G+\tilde{C}_{2}^H\approx(1+\alpha_{2}^G)\tilde{C}_{2}^G$, we distinguish the following four cases: $\alpha_{2}^G\geq0$, $-1<\alpha_{2}^G<0$, $\alpha_{2}^G=-1$ and  $\alpha_{2}^G<-1$. Let us suppose first that $-1<\alpha_{2}^G<0$. For this case, we have that $\tilde{C}_{2}^F>0$ and $\tilde{C}_{2}^H<0$. Now, consider $h_2^{F}$ and $h_2^{H}$ such that $0<h_2^{F}-\tau\ll 1$, $0<\tau - h_2^{H}\ll 1$ and
	\begin{equation*}
	\begin{aligned}
	\tilde{C}_{2}^F(h_{2}^F-\tau)^{2}&>0,\\
	\tilde{C}_{2}^H(h_{2}^H-\tau)^{2}&<0.
		\end{aligned}
\end{equation*} 
	Take $\tilde{C}_{1}^G$ and $\tilde{C}_{1}^H$ such that $|\tilde{C}_{1}^G|\ll\tilde{C}_{2}^G$, $\tilde{C}_{1}^G\tilde{C}_{2}^G>0$, $|\tilde{C}_{1}^H|\ll\tilde{C}_{2}^H$, $\tilde{C}_{1}^H\tilde{C}_{2}^H>0$, $\tilde{C}_{1}^G<|\tilde{C}_{1}^H|$, $|\tilde{C}_{1}^G+\tilde{C}_{1}^H|\ll\tilde{C}_{2}^F$ and
	\begin{equation*}
		\begin{aligned}
			\tilde{C}_{1}^G(h_{2}^G-\tau)+\tilde{C}_{2}^G(h_{2}^G-\tau)^2&>0,\\
			\tilde{C}_{1}^H(h_{2}^H-\tau)+\tilde{C}_{2}^H(h_{2}^H-\tau)^{2}&<0.
		\end{aligned}
	\end{equation*}
	As $\tilde{C}_{1}^G>0$, $\tilde{C}_{1}^H<0$, $|\tilde{C}_{1}^G+\tilde{C}_{1}^H|\ll\tilde{C}_{2}^F$ and $\tilde{C}_{1}^F=\tilde{C}_{1}^G+\tilde{C}_{1}^H<0$ then $|\tilde{C}_{1}^F|\ll\tilde{C}_{2}^F$, $\tilde{C}_{1}^F\tilde{C}_{2}^F<0$ and 
	$$\tilde{C}_{1}^F(h_{2}^F-\tau)+\tilde{C}_{2}^F(h_{2}^F-\tau)^2>0.$$
	Now, as $\tilde{C}_{1}^F<0$, $\tilde{C}_{1}^G>0$ and $\tilde{C}_{1}^H<0$, we can choose $h_1^{F}$, $h_1^{G}$ and $h_1^{H}$ such that $0<h_1^{F}-\tau\ll h_2^{F}-\tau\ll 1$, $0<\tau-h_1^{G}\ll\tau-h_2^{G}\ll 1$,  $0<\tau-h_1^{H}\ll\tau-h_2^{H}\ll 1$ and 
	\begin{equation*}
		\begin{aligned}
			\tilde{C}_{1}^F(h_{1}^F-\tau)+\tilde{C}_{2}^F(h_{1}^F-\tau)^2&<0,\\
			\tilde{C}_{1}^G(h_{1}^G-\tau)+\tilde{C}_{2}^G(h_{1}^G-\tau)^2&<0,\\
			\tilde{C}_{1}^H(h_{1}^H-\tau)+\tilde{C}_{2}^H(h_{1}^H-\tau)^2&>0
		\end{aligned}
	\end{equation*}
	Therefore, the equations
	\begin{equation*}
		\begin{aligned}
			\tilde{C}_{1}^F(h-\tau)+\tilde{C}_{2}^F(h-\tau)^2&=0,\\
			\tilde{C}_{1}^G(h-\tau)+\tilde{C}_{2}^G(h-\tau)^2&=0,\\
			\tilde{C}_{1}^H(h-\tau)+\tilde{C}_{2}^H(h-\tau)^2&=0,
		\end{aligned}
	\end{equation*}
	have roots $\tilde{h}_2^F$, $\tilde{h}_2^G$ and $\tilde{h}_2^H$, respectively, with  $h_1^F<\tilde{h}_2^F<h_2^F$, $h_2^G<\tilde{h}_2^G<h_1^G$ and $h_2^H<\tilde{h}_2^H<h_1^H$.
	
	Take $\tilde{C}_{0}^G$ and $\tilde{C}_{0}^H$ such that $|\tilde{C}_{0}^G|\ll|\tilde{C}_{1}^G|$, $\tilde{C}_{0}^G\tilde{C}_{1}^G>0$, $|\tilde{C}_{0}^H|\ll|\tilde{C}_{1}^H|$, $\tilde{C}_{0}^H\tilde{C}_{1}^H>0$, $|\tilde{C}_{0}^G+\tilde{C}_{0}^H|\leq|\tilde{C}_{1}^F|$, $|\tilde{C}_{0}^H|<\tilde{C}_{0}^G$ and 	
\begin{equation*}
	\begin{aligned}
		\tilde{C}_{0}^G+\tilde{C}_{1}^G(h_1^G-\tau)+\tilde{C}_{2}^G(h_1^G-\tau)^2&<0,\\
		\tilde{C}_{0}^H+\tilde{C}_{1}^H(h_1^H-\tau)+\tilde{C}_{2}^H(h_{1}^H-\tau)^2&>0.
	\end{aligned}
\end{equation*}
Now, as $\tilde{C}_{0}^G>0$, $\tilde{C}_{0}^H<0$, $|\tilde{C}_{0}^G+\tilde{C}_{0}^H|\leq|\tilde{C}_{1}^F|$ and $\tilde{C}_{0}^F=\tilde{C}_{0}^G+\tilde{C}_{0}^H>0$, then $|\tilde{C}_{0}^F|\ll|\tilde{C}_{1}^F|$, $\tilde{C}_{0}^F\tilde{C}_{1}^F<0$ and 
$$\tilde{C}_{0}^F+\tilde{C}_{1}^F(h_1^F-\tau)+\tilde{C}_{2}^F(h_1^F-\tau)^2<0.$$
Moreover, we can choose $h_0^{F}$, $h_0^{G}$ and $h_0^{H}$  such that $0<h_0^{F}-\tau\ll h_1^{F}-\tau\ll 1$, $0<\tau-h_0^{G}\ll\tau-h_1^{G}\ll 1$, $0<\tau-h_0^{H}\ll\tau-h_1^{H}\ll 1$ and 
\begin{equation*}
	\begin{aligned}
		\tilde{C}_{0}^F+\tilde{C}_{1}^F(h_0^F-\tau)+\tilde{C}_{2}^F(h_0^F-\tau)^2&>0,\\
		\tilde{C}_{0}^G+\tilde{C}_{1}^G(h_0^G-\tau)+\tilde{C}_{2}^G(h_0^G-\tau)^2&>0,\\
		\tilde{C}_{0}^H+\tilde{C}_{1}^H(h_0^H-\tau)+\tilde{C}_{2}^H(h_0^H-\tau)^2&<0.
	\end{aligned}
\end{equation*}
Therefore, the equations
\begin{equation*}
	\begin{aligned}
		\tilde{C}_{0}^F+\tilde{C}_{1}^F(h-\tau)+\tilde{C}_{2}^F(h-\tau)^2&=0,\\
		\tilde{C}_{0}^G+\tilde{C}_{1}^G(h-\tau)+\tilde{C}_{2}^G(h-\tau)^2&=0,\\
		\tilde{C}_{0}^H+\tilde{C}_{1}^H(h-\tau)+\tilde{C}_{2}^H(h-\tau)^2&=0,
	\end{aligned}
\end{equation*}
have roots $\tilde{h}_1^F$, $\tilde{h}_1^G$ and $\tilde{h}_1^H$, respectively, such that $h_0^F<\tilde{h}_1^F<h_1^F<\tilde{h}_2^F<h_2^F$,  $h_2^G<\tilde{h}_2^G<h_1^G<\tilde{h}_1^G<h_0^G$ and $h_2^H<\tilde{h}_2^H<h_1^H<\tilde{h}_1^H<h_0^H$. Note that we could have done all of the above computes by adding at each step the remainders of the functions \eqref{func:F}, \eqref{func:G} and \eqref{func:H} (indicated by the Landau Symbols). Thus, when $-1<\alpha_{2}^G<0$, we have the configuration of zeros $(2,2,2)$. 

The proof for the case $\alpha_{2}^G<-1$ is analogous to the previous one. Note that  $\tilde{C}_{2}^F<0$, and so we must choose the appropriate coefficients  $\tilde{C}_{i}^G>0$, $\tilde{C}_{i}^H<0$, $i=0,1,2$, $\tilde{C}_{1}^F>0$ and $\tilde{C}_{0}^F<0$. 

\medskip

Now, suppose that $\alpha_2^G=-1$. For this case, we have that $\tilde{C}_{2}^F=\mathcal{O}_1\big(\tilde{C}_{0}^G,\tilde{C}_{1}^G,\tilde{C}_{0}^H,\tilde{C}_{1}^H\big)$ and $\tilde{C}_{2}^H<0$. Hence, we cannot measure $|\tilde{C}_{2}^F|$ fixing only the value of $\tilde{C}_{2}^G$. So let us disregard the coefficient $\tilde{C}_{2}^F$ in the process of changing signs in the values of the image of the function \eqref{func:F}. Let $h_2^H$ such that $0<\tau-h_2^H\ll1$ and 
$$\tilde{C}_{2}^H(h_{2}^H-\tau)^{2}<0.$$ 
Take $\tilde{C}_{1}^G$ and $\tilde{C}_{1}^H$ such that $|\tilde{C}_{1}^G|\ll\tilde{C}_{2}^G$, $\tilde{C}_{1}^G\tilde{C}_{2}^G>0$, $|\tilde{C}_{1}^H|\ll|\tilde{C}_{2}^H|$, $\tilde{C}_{1}^H\tilde{C}_{2}^H>0$, $|\tilde{C}_{1}^H|<\tilde{C}_{1}^G$ and 
	\begin{equation*}
	\begin{aligned}
		\tilde{C}_{1}^G(h_2^G-\tau)+\tilde{C}_{2}^G(h_2^G-\tau)^2&>0,\\
		\tilde{C}_{1}^H(h_2^H-\tau)+\tilde{C}_{2}^H(h_2^H-\tau)^2&<0.
	\end{aligned}
\end{equation*}
As $\tilde{C}_{1}^G+\tilde{C}_{1}^H>0$ and $\tilde{C}_{1}^G+\tilde{C}_{1}^H=\tilde{C}_{1}^F$, then we have that $\tilde{C}_{1}^F>0$. Let $h_1^F$ such that $0<h_1^F-\tau\ll1$ and 
$$\tilde{C}_{1}^F(h_{1}^F-\tau)>0.$$ 
Now, as $\tilde{C}_{1}^G>0$ and $\tilde{C}_{1}^H<0$, we can choose $h_1^G$ and $h_1^H$ such that $0<\tau-h_1^G\ll\tau-h_2^G\ll1$, $0<\tau-h_1^H\ll\tau-h_2^H\ll1$ and
	\begin{equation*}
	\begin{aligned}
		\tilde{C}_{1}^G(h_1^G-\tau)+\tilde{C}_{2}^G(h_1^G-\tau)^2&<0,\\
		\tilde{C}_{1}^H(h_1^H-\tau)+\tilde{C}_{2}^H(h_1^H-\tau)^2&>0.
	\end{aligned}
\end{equation*} 
Therefore, the equations 
	\begin{equation*}
	\begin{aligned}
		\tilde{C}_{1}^G(h-\tau)+\tilde{C}_{2}^G(h-\tau)^2&=0,\\
		\tilde{C}_{1}^H(h-\tau)+\tilde{C}_{2}^H(h-\tau)^2&=0,
	\end{aligned}
\end{equation*}
have roots $\tilde{h}_2^G$ and $\tilde{h}_2^H$, respectively, such that $h_2^G<\tilde{h}_2^G<h_1^G$ and $h_2^H<\tilde{h}_2^H<h_1^H$.

Take $\tilde{C}_{0}^G$ and $\tilde{C}_{0}^H$ such that $|\tilde{C}_{0}^G|\ll\tilde{C}_{1}^G$, $\tilde{C}_{0}^G\tilde{C}_{1}^G>0$, $|\tilde{C}_{0}^H|\ll\tilde{C}_{1}^H$, $\tilde{C}_{0}^H\tilde{C}_{1}^H>0$, $|\tilde{C}_{0}^G+\tilde{C}_{0}^H|\ll|\tilde{C}_{1}^F|$, $\tilde{C}_{0}^G<|\tilde{C}_{0}^H|$  and
	\begin{equation*}
	\begin{aligned}
		\tilde{C}_{0}^G+\tilde{C}_{1}^G(h_1^G-\tau)+\tilde{C}_{2}^G(h_1^G-\tau)^2&<0,\\
		\tilde{C}_{0}^H+\tilde{C}_{1}^H(h_1^H-\tau)+\tilde{C}_{2}^H(h_1^H-\tau)^2&>0.
	\end{aligned}
\end{equation*} 
As  $\tilde{C}_{0}^G<|\tilde{C}_{0}^H|$, $|\tilde{C}_{0}^G+\tilde{C}_{0}^H|\ll\tilde{C}_{1}^F$ and $\tilde{C}_{0}^G+\tilde{C}_{0}^H=\tilde{C}_{0}^F$ then $|\tilde{C}_{0}^F|\ll\tilde{C}_{1}^F$, $\tilde{C}_{0}^F\tilde{C}_{1}^F<0$ and 
$$\tilde{C}_{0}^F+\tilde{C}_{1}^F(h_1^F-\tau)>0.$$
Now, as $\tilde{C}_{0}^F<0$, $\tilde{C}_{0}^G>0$ and $\tilde{C}_{0}^H<0$, we can choose $h_0^F$, $h_0^G$ and $h_0^H$ such that  $0<h_0^F-\tau\ll h_1^F-\tau\ll1$, $0<\tau-h_0^G\ll\tau-h_1^G\ll1$, $0<\tau-h_0^H\ll\tau-h_1^H\ll1$ and
\begin{equation*}
	\begin{aligned}
		\tilde{C}_{0}^F+\tilde{C}_{1}^F(h_0^F-\tau)&<0,\\
		\tilde{C}_{0}^G+\tilde{C}_{1}^G(h_0^G-\tau)+\tilde{C}_{2}^G(h_0^G-\tau)^2&>0,\\
		\tilde{C}_{0}^H+\tilde{C}_{1}^H(h_0^H-\tau)+\tilde{C}_{2}^H(h_0^H-\tau)^2&<0.
	\end{aligned}
\end{equation*}
Therefore, the equations
\begin{equation*}
	\begin{aligned}
		\tilde{C}_{0}^F+\tilde{C}_{1}^F(h-\tau)&=0,\\
		\tilde{C}_{0}^G+\tilde{C}_{1}^G(h-\tau)+\tilde{C}_{2}^G(h-\tau)^2&=0,\\
		\tilde{C}_{0}^H+\tilde{C}_{1}^H(h-\tau)+\tilde{C}_{2}^H(h-\tau)^2&=0,
	\end{aligned}
\end{equation*}
have roots $\tilde{h}_1^F$, $\tilde{h}_1^G$ and $\tilde{h}_1^H$, respectively, such that $h_0^F<\tilde{h}_1^F<h_1^F$, $h_2^G<\tilde{h}_2^G<h_1^G<\tilde{h}_1^G<h_0^G$ and $h_2^H<\tilde{h}_2^H<h_1^H<\tilde{h}_1^H<h_0^H$. Then, when $\alpha_{2}^G=-1$, we have the configuration of zeros $(1,2,2)$. 

The proof for the case $\alpha_{2}^G\geq0$ is analogous to the previous one. Observe that  $\tilde{C}_{2}^F>0$ and $\tilde{C}_{2}^H>0$ (or $\tilde{C}_{2}^H=\mathcal{O}_1\big(\tilde{C}_{0}^G,\tilde{C}_{1}^G,\tilde{C}_{0}^H,\tilde{C}_{1}^H\big)$). Hence, we cannot be sure that $\tilde{C}_{2}^H<0$ by setting $\tilde{C}_{2}^G>0$. Therefore, similar to the previous case, we disregard the coefficient $\tilde{C}_{2}^H$ in the process of changing signs in the values of the image of the function \eqref{func:H}. Choose the appropriate coefficients  $\tilde{C}_{i}^G>0$, $\tilde{C}_{i}^H<0$, $i=0,1$, $\tilde{C}_{1}^F<0$ and $\tilde{C}_{0}^F>0$, we have the configuration of zeros $(2,2,1)$ for this case. 
\end{proof}


\medskip  

The next lemma is similar to Lemma 11 from \cite{Pes22b} and its proof is analogous. 

\begin{lemma}\label{lem:01}
	Let $f:\mathbb{R}\rightarrow\mathbb{R}$ be a function of class $\mathbb{C}^{k+1}$, $k\geq 1$. Then 
	\begin{eqnarray}
		&& f(x)\log(\tau-x)=\mathcal{P}(x)\log(\tau-x)+\mathcal{R}(x),\quad\text{when}\quad x<\tau,\label{cond1:lem}\\
		&& f(x)\log(x-\tau)=\mathcal{P}(x)\log(x-\tau)+\mathcal{R}(x),\quad\text{when}\quad x>\tau,\label{cond2:lem}
	\end{eqnarray}
	where $\mathcal{P}(x)=\sum_{i=0}^{k}\frac{f^{(i)}(\tau)}{i!}(x-\tau)^i$ is the Taylor’s polynomial of $f$ at $x=\tau$ of degree $k$ and $\mathcal{R}(x)=\log(\pm(\tau-x))(f(x)-\mathcal{P}(x))$. Moreover, $\lim_{x\rightarrow \tau}\mathcal{R}(x)=0$.
\end{lemma}

\bigskip


\begin{proof}[Proof of Theorem \ref{the:01}] In order to prove the Theorem \ref{the:01}, let us consider the Melnikov functions $M_0(h)$, $M_1(h)$ and $M_2(h)$ associated to system \eqref{eq:01}  given by Theorems \ref{theo:sss}, \ref{theo:css} and \ref{theo:csc}. We will distinguish the prove in three cases. We begin with the case where the system $\eqref{eq:01}|_{\epsilon=0}$ is of the type SSS. For this case, we have two sub-cases. The first one is when $\tau_{\scriptscriptstyle LS}\ne\tau_{\scriptscriptstyle RS}$ (i.e. we have a homoclinic loop) and the second one is when $\tau_{\scriptscriptstyle LS}=\tau_{\scriptscriptstyle RS}$ (i.e. we have two heteroclinic orbits). The second and third cases are when the system $\eqref{eq:01}|_{\epsilon=0}$ is of the type CSS and CSC, respectively. 
	
\bigskip
	
\noindent{\it Case SSS.}  Consider the Melnikov functions given by the Theorem \ref{theo:sss}. When we have homoclinic loop, i.e. $\tau_{\scriptscriptstyle LS}\ne\tau_{\scriptscriptstyle RS}$, by Lemma \ref{lem:01}, we can expand the functions \eqref{mel:sss00}--\eqref{mel:sss20} at $h=1$  as
\begin{eqnarray}
&& M_0(h)=\sum_{j=0}^{2}C_{0}^j(h-1)^j+\sum_{j=1}^{2}D_{0}^j\log(h-1)(h-1)^j + \mathcal{O}((h-1)^3),\label{exp:sssm01}\\
&& M_1(h)=\sum_{j=0}^{2}C_{1}^j(h-1)^j+\sum_{j=1}^{2}D_{1}^j\log(1-h)(h-1)^j + \mathcal{O}((h-1)^3),\label{exp:sssm11}\\
&& M_2(h)=\sum_{j=0}^{2}C_{2}^j(h-1)^j+\sum_{j=1}^{2}D_{2}^j\log(1-h)(h-1)^j + \mathcal{O}((h-1)^3),\label{exp:sssm21}
\end{eqnarray}
where
\begin{equation*}
	\begin{aligned}
		& C_{0}^0 = \frac{2b_{\scriptscriptstyle R}}{b_{\scriptscriptstyle L}}(r_{10}-r_{00})+2 (p_{00} + p_{10} - b_{\scriptscriptstyle R} u_{10} +b_{\scriptscriptstyle R} v_{01}) + \frac{b_{\scriptscriptstyle R}\tau_{\scriptscriptstyle LS}}{\omega_{\scriptscriptstyle LS}}(r_{10} + s_{01}) + \frac{b_{\scriptscriptstyle R}\tau_{\scriptscriptstyle RS}}{\omega_{\scriptscriptstyle RS}}\\
		&\quad\quad\quad(p_{10} + q_{01}) - \frac{b_{\scriptscriptstyle R}}{2\omega_{\scriptscriptstyle LS}}(r_{10} + s_{01})(\tau_{\scriptscriptstyle LS}^2-1)\log\bigg(\frac{\tau_{\scriptscriptstyle LS}+1}{\tau_{\scriptscriptstyle LS}-1}\bigg) - \frac{b_{\scriptscriptstyle R}}{2\omega_{\scriptscriptstyle RS}}(p_{10} + q_{01}) \\
		&\quad\quad\quad (\tau_{\scriptscriptstyle RS}^2-1)\log\bigg(\frac{\tau_{\scriptscriptstyle RS}+1}{\tau_{\scriptscriptstyle RS}-1}\bigg),\\
		& C_{0}^1 =\frac{2b_{\scriptscriptstyle R}}{b_{\scriptscriptstyle L}}(r_{10}-r_{00})+2 (p_{00} + p_{10} - b_{\scriptscriptstyle R} u_{10} +b_{\scriptscriptstyle R} v_{01})+b_{\scriptscriptstyle R} (u_{10} + v_{01}) \log(4) + \frac{b_{\scriptscriptstyle R}}{\omega_{\scriptscriptstyle LS}}\\
		&\quad\quad\quad(r_{10} + s_{01})\log\bigg(\frac{\tau_{\scriptscriptstyle LS}+1}{\tau_{\scriptscriptstyle LS}-1}\bigg) +  \frac{b_{\scriptscriptstyle R}}{\omega_{\scriptscriptstyle RS}}(p_{10} + q_{01})\log\bigg(\frac{\tau_{\scriptscriptstyle RS}+1}{\tau_{\scriptscriptstyle RS}-1}\bigg),
		\end{aligned}
	\end{equation*}
\begin{equation*}
	\begin{aligned}
		& C_{0}^2 =\frac{b_{\scriptscriptstyle R}}{2}\bigg( 2 u_{10} + 2\bigg( v_{01} + \frac{\tau_{\scriptscriptstyle LS}}{\omega_{\scriptscriptstyle LS}(\tau_{\scriptscriptstyle LS}^2-1)}(r_{10} + s_{01}) +  \frac{\tau_{\scriptscriptstyle RS}}{\omega_{\scriptscriptstyle RS}(\tau_{\scriptscriptstyle RS}^2-1)}(p_{10} + q_{01}) \bigg) \\
		&\quad\quad\quad+(u_{10} + v_{01}) \log(4) + \frac{1}{\omega_{\scriptscriptstyle LS}}(r_{10} + s_{01})\log\bigg(\frac{\tau_{\scriptscriptstyle LS}+1}{\tau_{\scriptscriptstyle LS}-1}\bigg)+\frac{1}{\omega_{\scriptscriptstyle RS}}(p_{10} + q_{01})\\
		&\quad\quad\quad\log\bigg(\frac{\tau_{\scriptscriptstyle RS}+1}{\tau_{\scriptscriptstyle RS}-1}\bigg)\bigg),
	\end{aligned}
	\end{equation*}
\begin{equation*}
	\begin{aligned}
	& C_{1}^0 = 2 p_{00} + 2 p_{10} + b_{\scriptscriptstyle R} (v_{01} - 2 u_{00} - u_{10} ) + \frac{b_{\scriptscriptstyle R}\tau_{\scriptscriptstyle RS}}{\omega_{\scriptscriptstyle RS}}(p_{10} + q_{01}) - \frac{b_{\scriptscriptstyle R}}{2\omega_{\scriptscriptstyle RS}}(p_{10} + q_{01})\,\,\\
	&\quad\quad\quad(\tau_{\scriptscriptstyle RS}^2-1)\log\bigg(\frac{\tau_{\scriptscriptstyle RS}+1}{\tau_{\scriptscriptstyle RS}-1}\bigg),\\
	& C_{1}^1 = 2 p_{00} + 2 p_{10} + b_{\scriptscriptstyle R} \Big(v_{01} - 2 u_{00} - u_{10} + (u_{10} + v_{01}) \log(2)\Big) +  \frac{b_{\scriptscriptstyle R}}{\omega_{\scriptscriptstyle RS}}(p_{10} + q_{01})\\
	&\quad\quad\quad\log\bigg(\frac{\tau_{\scriptscriptstyle RS}+1}{\tau_{\scriptscriptstyle RS}-1}\bigg),
	\end{aligned}
\end{equation*}
\begin{equation*}
	\begin{aligned}
	& C_{1}^2 = \frac{1}{2}\bigg(\frac{2\tau_{\scriptscriptstyle RS}}{\omega_{\scriptscriptstyle RS}(\tau_{\scriptscriptstyle RS}^2-1)}(p_{10} + q_{01}) + (u_{10} + v_{01}) (1 + \log(2)) +  \frac{1}{\omega_{\scriptscriptstyle RS}}(p_{10} + q_{01})\quad\\
	&\quad\quad\quad \log\bigg(\frac{\tau_{\scriptscriptstyle RS}+1}{\tau_{\scriptscriptstyle RS}-1}\bigg)\bigg),\\
	& C_{2}^0 = -\frac{2}{b_{\scriptscriptstyle L}}(r_{00} - r_{10})+2 u_{00} - u_{10} + v_{01}+\frac{\tau_{\scriptscriptstyle LS}}{\omega_{\scriptscriptstyle LS}}(r_{10} + s_{01}) + \frac{1}{2\omega_{\scriptscriptstyle LS}}(r_{10} + s_{01})\\
	&\quad\quad\quad (\tau_{\scriptscriptstyle LS}^2-1)\log\bigg(\frac{\tau_{\scriptscriptstyle LS}+1}{\tau_{\scriptscriptstyle LS}-1}\bigg),
		\end{aligned}
\end{equation*}
\begin{equation*}
\begin{aligned}
	& C_{2}^1 = -\frac{2}{b_{\scriptscriptstyle L}}(r_{00} - r_{10})+2 u_{00} - u_{10} + v_{01} + (u_{10} + v_{01}) \log(2) + \frac{1}{\omega_{\scriptscriptstyle LS}}(r_{10} + s_{01})\\
	&\quad\quad\quad \log\bigg(\frac{\tau_{\scriptscriptstyle LS}+1}{\tau_{\scriptscriptstyle LS}-1}\bigg),\\
	& C_{2}^2 = \frac{1}{2}\bigg(\frac{2\tau_{\scriptscriptstyle LS}}{\omega_{\scriptscriptstyle LS}(\tau_{\scriptscriptstyle LS}^2-1)}(r_{10} + s_{01}) + (u_{10} + v_{01}) (1 + \log(2)) +  \frac{1}{\omega_{\scriptscriptstyle LS}}(r_{10} + s_{01})\quad\\
	&\quad\quad\quad \log\bigg(\frac{\tau_{\scriptscriptstyle LS}+1}{\tau_{\scriptscriptstyle LS}-1}\bigg)\bigg),
	\end{aligned}
\end{equation*}
and $D_{i}^k$, for $i=0,1,2$ and $k=1,2$, depending on the parameters of system \eqref{eq:01}, whose expressions have been omitted for simplicity. Now, when we have two    heteroclinic orbits, i.e. $\tau_{\scriptscriptstyle LS}=\tau_{\scriptscriptstyle RS}$, the expansions of Melnikov functions are the same as in \eqref{exp:sssm01}--\eqref{exp:sssm21} with $C_{i}^j=C_{i}^j\mid_{\tau_{\scriptscriptstyle LS}=\tau_{\scriptscriptstyle RS}}$ and $D_{i}^j=D_{i}^j\mid_{\tau_{\scriptscriptstyle LS}=\tau_{\scriptscriptstyle RS}}$, $i=0,1,2$, $j=0,1,2,3$. Moreover, as 
$$\lim_{h\rightarrow 1}\log(h-1)(h-1)^j=0,\quad j=1,2,\quad\text{for}\quad 0<h-1,$$
$$\lim_{h\rightarrow 1}\log(1-h)(h-1)^j=0,\quad j=1,2,\quad\text{for}\quad 0<1-h,$$
the factors $\log(\pm(h-1))$ of the expansions in \eqref{exp:sssm01}--\eqref{exp:sssm21}  can be disregarded in the study of the number of zeros. Therefore, will consider only the coefficients $C_{i}^k$, for $i=0,1,2$ and $k=0,1,2,3$.

\medskip

Further, when $\tau_{\scriptscriptstyle LS}\ne\tau_{\scriptscriptstyle RS}$, we have that 
\begin{equation}\label{eq:coefdifsss}
	\begin{aligned}
		& C_{0}^0 -(C_{1}^0+C_{2}^0) = \bigg(\frac{b_{\scriptscriptstyle R}-1}{2b_{\scriptscriptstyle L}\omega_{\scriptscriptstyle LS}}\bigg)\bigg(4\omega_{\scriptscriptstyle LS}(r_{10} - r_{00} ) + 
		2 b_{\scriptscriptstyle L} (2 u_{00} - u_{10} + v_{01})\omega_{\scriptscriptstyle LS}   \\
		&\quad\quad \quad\quad\quad\quad\quad\quad+ 2 b_{\scriptscriptstyle L} (r_{10} + s_{01}) \tau_{\scriptscriptstyle LS} - b_{\scriptscriptstyle L} (r_{10} + s_{01}) ( \tau_{\scriptscriptstyle LS}^2-1)\log\bigg(\frac{\tau_{\scriptscriptstyle LS}+1}{\tau_{\scriptscriptstyle LS}-1}\bigg)\bigg),\\
		& C_{0}^1 -(C_{1}^1+C_{2}^1) = \bigg(\frac{b_{\scriptscriptstyle R}-1}{b_{\scriptscriptstyle L}\omega_{\scriptscriptstyle LS}}\bigg)\bigg(\omega_{\scriptscriptstyle LS} \Big(2( r_{10}-r_{00}) + 
		b_{\scriptscriptstyle L} \Big(2 u_{00} - u_{10} + v_{01}  \\
		&\quad\quad\quad \quad\quad\quad\quad\quad+ (u_{10} + v_{01})\log(2)\Big)\Big) + b_{\scriptscriptstyle L} (r_{10} + s_{01})\log\bigg(\frac{\tau_{\scriptscriptstyle LS}+1}{\tau_{\scriptscriptstyle LS}-1}\bigg) \bigg),\\
		& C_{0}^2 -(C_{1}^2+C_{2}^2) = \bigg(\frac{b_{\scriptscriptstyle R}-1}{2\omega_{\scriptscriptstyle LS}(\tau_{\scriptscriptstyle LS}^2-1)}\bigg)\bigg(2 (r_{10} + s_{01}) \tau_{\scriptscriptstyle LS} + (u_{10} + v_{01}) \omega_{\scriptscriptstyle LS} (\tau_{\scriptscriptstyle LS}^2-1) \\
		&\quad\quad\quad \quad\quad\quad\quad\quad(1 +\log(2)) + (r_{10} + s_{01}) ( \tau_{\scriptscriptstyle LS}^2-1)\log\bigg(\frac{\tau_{\scriptscriptstyle LS}+1}{\tau_{\scriptscriptstyle LS}-1}\bigg)\bigg).\\
	\end{aligned}
\end{equation}
Thus, if $b_{\scriptscriptstyle R}=1$, we have that $C_{0}^0=C_{1}^0+C_{2}^0$, $C_{0}^1=C_{1}^1+C_{2}^1$ and $C_{0}^2=C_{1}^2+C_{2}^2$. 
Now, when
\begin{equation*}
	\begin{aligned}
		& p_{00} = \frac{1}{4\omega_{\scriptscriptstyle RS}}\bigg(2 (u_{10} - v_{01} + 2( u_{00} - p_{10}) ) \omega_{\scriptscriptstyle RS} - 2 (p_{10} + q_{01}) \tau_{\scriptscriptstyle RS} + (p_{10} + q_{01}) (\tau_{\scriptscriptstyle RS}^2-1 )\\
		&\quad\quad\quad\log\bigg(\frac{\tau_{\scriptscriptstyle RS}+1}{\tau_{\scriptscriptstyle RS}-1}\bigg)\bigg),\\
		& u_{10} = -\frac{1}{\omega_{\scriptscriptstyle RS}\log(4)}\bigg(-2 (p_{10} + q_{01}) \tau_{\scriptscriptstyle RS} + v_{01} \omega_{\scriptscriptstyle RS} \log(4) + (p_{10} + q_{01}) (1 + \tau_{\scriptscriptstyle RS}^2)\\
		&\quad\quad\quad\log\bigg(\frac{\tau_{\scriptscriptstyle RS}+1}{\tau_{\scriptscriptstyle RS}-1}\bigg)\bigg),\\
		& p_{10} = - q_{01},\\
		& r_{00} = \frac{1}{4\omega_{\scriptscriptstyle LS}}\bigg(4 (r_{10} + b_{\scriptscriptstyle L} (u_{00} + v_{01})) \omega_{\scriptscriptstyle RS} + 
		2 b_{\scriptscriptstyle L} (r_{10} + s_{01}) \tau_{\scriptscriptstyle LS} - 
		b_{\scriptscriptstyle L} (r_{10} + s_{01}) ( \tau_{\scriptscriptstyle LS}^2-1)\\
		&\quad\quad\quad\log\bigg(\frac{\tau_{\scriptscriptstyle LS}+1}{\tau_{\scriptscriptstyle LS}-1}\bigg)\bigg),\\
		& r_{10} = - s_{01},
	\end{aligned}
\end{equation*}
we have that 
$$C_{1}^0=C_{1}^1=C_{1}^2=C_{2}^0=C_{2}^1=C_{2}^2=0,\quad\text{with}\quad\text{rank}\frac{\partial(C_{1}^0,C_{1}^1,C_{1}^2,C_{2}^0, C_{2}^1)}{\partial(p_{00},u_{10},p_{10},r_{00},r_{10})}=5.$$ 
Moreover, $C_{2}^2=\alpha_1^2 C_{1}^2+\alpha_1^0 C_{1}^0+\alpha_1^1 C_{1}^1+\alpha_2^0 C_{2}^0+\alpha_2^1 C_{2}^1$, where $\alpha_1^2$ correspond to $\alpha_2^G$ from proof of Proposition \ref{the:coef} and it is given by
$$\alpha_1^2=\frac{\phi(\tau_{\scriptscriptstyle RS})}{\phi(\tau_{\scriptscriptstyle LS})},\quad \text{with}\quad \phi(\tau)=\frac{\lambda_1(\tau)}{\lambda_2(\tau)},$$
where
\begin{equation*}
	\begin{aligned}
		&\lambda_1(\tau)=(\tau^2-1)\bigg((\tau^2+1)\log\bigg(\dfrac{\tau+1}{\tau-1}\bigg)-2\tau\bigg),\\
		&\lambda_2(\tau)=\tau\Big(2-\log(4)-\tau^2(2+\log(4))\Big)+\Big(\log(2)-1+\tau^4(1+\log(2))-\tau^2\log(4)\Big)\\
		& \quad\quad\quad\,\,\log\bigg(\dfrac{\tau+1}{\tau-1}\bigg).
	\end{aligned}
\end{equation*}

\medskip

For $a_{\scriptscriptstyle R}=c_{\scriptscriptstyle L}=0$, $b_{\scriptscriptstyle R}=c_{\scriptscriptstyle R}=a_{\scriptscriptstyle L}=1$, $\beta_{\scriptscriptstyle R}=-3$, $b_{\scriptscriptstyle L}\in\{2/3,1\}$ and $\beta_{\scriptscriptstyle L}\in\{2,3\}$, the right and left subsystems from $\eqref{eq:01}|_{\epsilon=0}$ have real saddles at the points $(3,0)$ and $(-3,3)$, respectively, with a homoclinic loop when $b_{\scriptscriptstyle L}=2/3$ and $\beta_{\scriptscriptstyle L}=3$. Now, for the same parameter values, except for $b_{\scriptscriptstyle L}=1$ and $\beta_{\scriptscriptstyle L}=2$, the right and left subsystems from $\eqref{eq:01}|_{\epsilon=0}$ have real saddles at the points $(3,0)$ and $(-3,2)$, respectively, with two heteroclinic orbits.
Moreover, in the homoclinic case, we have that $\alpha_1^2\approx 0.80487$, and in the  heteroclinic case, we have that $\alpha_1^2= 1$. This implies, by Proposition \ref{the:coef}, that $C_{1}^0$, $C_{1}^1$, $C_{1}^2$, $C_{2}^0$ and $C_{2}^1$ can be taken as free coefficients, such that $\max \{N_{M_0}+N_{M_1}+N_{M_2}\}$ is at least five in a neighborhood of $h=1$. More precisely, the configuration of zeros $(2,2,1)$ is achievable (obviously we can obtain the configuration $(2, 1, 2)$ through a reflection with respect to the $y$ axis). Therefore, the number of limit cycles from system \eqref{eq:01} that can bifurcate of the period annulus near $h=1$, when system $\eqref{eq:01}|_{\epsilon=0}$ is of the type SSS, is at least five.



\bigskip

\noindent{\it Case CSS.}  Consider the Melnikov functions given by the Theorem \ref{theo:css}. By Lemma \ref{lem:01} we can expand these functions at $h=1$ as
\begin{eqnarray*}
	&& M_0(h)=\sum_{j=0}^{3}C_{0}^j(h-1)^j+\sum_{j=1}^{2}D_{0}^j\log(h-1)(h-1)^j + \mathcal{O}((h-1)^4),\label{exp:cssm0}\\
	&& M_1(h)=\sum_{j=0}^{2}C_{1}^j(h-1)^j+\sum_{j=1}^{2}D_{1}^j\log(1-h)(h-1)^j + \mathcal{O}((h-1)^3),\label{exp:cssm1}\\
	&& M_2(h)=\sum_{j=0}^{2}C_{2}^j(h-1)^j+\sum_{j=1}^{2}D_{2}^j\log(1-h)(h-1)^j + \mathcal{O}((h-1)^3),\label{exp:cssm2}
\end{eqnarray*}
where $C_{i}^j$ and $D_{i}^j$, $i=0,1,2$ and $j=0,\dots,3$, depending on the parameters of system \eqref{eq:01}. As in the previous case, we will consider only the coefficients $C_{i}^j$, $i=0,1,2$ and $j=0,\dots,3$, whose expressions are in Appendix A.

\medskip

Further, similarly as in \eqref{eq:coefdifsss}, we have that if $b_{\scriptscriptstyle R}=1$ then $C_{0}^0=C_{1}^0+C_{2}^0$, $C_{0}^1=C_{1}^1+C_{2}^1$ and $C_{0}^2=C_{1}^2+C_{2}^2$. Now, when
\begin{equation*}
	\begin{aligned}
		& p_{00} = \frac{1}{4\omega_{\scriptscriptstyle RS}}\bigg(2 (u_{10} - v_{01} + 2( u_{00} - p_{10}) ) \omega_{\scriptscriptstyle RS} - 2 (p_{10} + q_{01}) \tau_{\scriptscriptstyle RS} + (p_{10} + q_{01}) (\tau_{\scriptscriptstyle RS}^2-1 )\\
		&\quad\quad\quad\log\bigg(\frac{\tau_{\scriptscriptstyle RS}+1}{\tau_{\scriptscriptstyle RS}-1}\bigg)\bigg),\\
		& u_{10} = -\frac{1}{\omega_{\scriptscriptstyle RS}\log(4)}\bigg(-2 (p_{10} + q_{01}) \tau_{\scriptscriptstyle RS} + v_{01} \omega_{\scriptscriptstyle RS} \log(4) + (p_{10} + q_{01}) (1 + \tau_{\scriptscriptstyle RS}^2)\\
		&\quad\quad\quad\log\bigg(\frac{\tau_{\scriptscriptstyle RS}+1}{\tau_{\scriptscriptstyle RS}-1}\bigg)\bigg),\\
		& p_{10} = - q_{01},\\
		& r_{00} = \frac{1}{4\omega_{\scriptscriptstyle LS}}\bigg(4 r_{10} \omega_{\scriptscriptstyle LC} + 4 b_{\scriptscriptstyle L} (u_{00} + v_{01}) \omega_{\scriptscriptstyle LC} + 2 b_{\scriptscriptstyle L} (r_{10} + s_{01}) (\pi \mu_2 (1 + \tau_{\scriptscriptstyle LC}^2)-\tau_{\scriptscriptstyle LC}) \\
		&\quad\quad\quad+ (-1)^{\mu_2} b_{\scriptscriptstyle L} (r_{10} + s_{01}) (1 + \tau_{\scriptscriptstyle LC}^2)\arccos\bigg(\frac{\tau_{\scriptscriptstyle LC}^2-1}{\tau_{\scriptscriptstyle LC}^2+1}\bigg)\bigg),\\
		& r_{10} = - s_{01},
	\end{aligned}
\end{equation*}
we have that 
$$C_{1}^0=C_{1}^1=C_{1}^2=C_{2}^0=C_{2}^1=C_{2}^2=0,\quad\text{with}\quad\text{rank}\frac{\partial(C_{1}^0,C_{1}^1,C_{1}^2,C_{2}^0,C_{2}^1)}{\partial(p_{00},u_{10},p_{10},r_{00},r_{10})}=5.$$
Moreover, $C_{2}^2=\alpha_1^2 C_{1}^2+\alpha_1^0 C_{1}^0+\alpha_1^1 C_{1}^1+\alpha_2^0 C_{2}^0+\alpha_2^1 C_{2}^1$. For this case, with $a_{\scriptscriptstyle R}=0$, $b_{\scriptscriptstyle R}=c_{\scriptscriptstyle R}=a_{\scriptscriptstyle L}=1$, $\beta_{\scriptscriptstyle R}=-3$, $b_{\scriptscriptstyle L}=2$, $c_{\scriptscriptstyle L}=-1$ and $\beta_{\scriptscriptstyle L}\in\{-4,-1/2\}$, the right subsystem from $\eqref{eq:01}|_{\epsilon=0}$ has a real saddle at the point $(3,0)$ and the left subsystem has a virtual (resp. real) center at the point $(0,-1/2)$ (resp. $(-7,3)$) when $\beta_{\scriptscriptstyle L}=-1/2$ (resp. $\beta_{\scriptscriptstyle L}=-4$). Moreover, we have that $\alpha_1^2\approx-0.2170$ (resp. $\alpha_1^2\approx-3.3874$). This implies, by Proposition \ref{the:coef}, that $C_{1}^0$, $C_{1}^1$, $C_{1}^2$, $C_{2}^0$ and $C_{2}^1$ can be taken as free coefficients, such that $\max \{N_{M_0}+N_{M_1}+N_{M_2}\}$ is at least six in a neighborhood of $h=1$. More precisely, the configuration of zeros $(2,2,2)$ is achievable. Therefore, the number of limit cycles from system \eqref{eq:01} that can bifurcate of the period annulus near $h=1$, when system $\eqref{eq:01}|_{\epsilon=0}$ is of the type CSS, is at least six.



\bigskip

\noindent{\it Case CSC.}  Consider the Melnikov functions given by the Theorem \ref{theo:csc}. By Lemma \ref{lem:01} we can expand these functions at $h=1$ as
\begin{eqnarray*}
	&& M_0(h)=\sum_{j=0}^{3}C_{0}^j(h-1)^j+\sum_{j=1}^{2}D_{0}^j\log(h-1)(h-1)^j + \mathcal{O}((h-1)^4),\label{exp:cscm0}\\
	&& M_1(h)=\sum_{j=0}^{2}C_{1}^j(h-1)^j+\sum_{j=1}^{2}D_{1}^j\log(1-h)(h-1)^j + \mathcal{O}((h-1)^3),\label{exp:cscm1}\\
	&& M_2(h)=\sum_{j=0}^{2}C_{2}^j(h-1)^j+\sum_{j=1}^{2}D_{2}^j\log(1-h)(h-1)^j + \mathcal{O}((h-1)^3),\label{exp:cscm2}
\end{eqnarray*}
where $C_{i}^j$ and $D_{i}^j$, $i=0,1,2$ and $j=0,\dots,3$, depending on the parameters of system \eqref{eq:01}. As in the previous cases, we will consider only the coefficients $C_{i}^j$, $i=0,1,2$ and $j=0,\dots,3$, whose expressions are in Appendix B.

\medskip

Further, similarly as in \eqref{eq:coefdifsss}, we have that if $b_{\scriptscriptstyle R}=1$ then $C_{0}^0=C_{1}^0+C_{2}^0$, $C_{0}^1=C_{1}^1+C_{2}^1$ and $C_{0}^2=C_{1}^2+C_{2}^2$. Now, when
\begin{equation*}
	\begin{aligned}
		& p_{00} = \frac{1}{4\omega_{\scriptscriptstyle RC}}\bigg(-2 \pi (p_{10} + q_{01}) \mu_1 + 2 \Big( 2 (u_{00} -p_{10}) + u_{10} - v_{01}\Big) \omega_{\scriptscriptstyle RC} + 2 (p_{10} + q_{01}) \tau_{\scriptscriptstyle RC} \\
		&\quad\quad\quad - 2 \pi (p_{10} + q_{01}) \mu_1 \tau_{\scriptscriptstyle RC}^2 + (-1)^{(1 + \mu_1)} (p_{10} + q_{01}) (1 + \tau_{\scriptscriptstyle RC}^2) \arccos\bigg(\frac{\tau_{\scriptscriptstyle RC}^2-1}{\tau_{\scriptscriptstyle RC}^2+1}\bigg)\bigg),\\
		& u_{10} = -\frac{1}{\omega_{\scriptscriptstyle RC}\log(4)}\bigg((p_{10} + q_{01}) \Big(-2 (-1)^{\mu_1} + (-1)^{\mu_1}  (1 + \tau_{\scriptscriptstyle RC}^2)\Big) \arccos\bigg(\frac{\tau_{\scriptscriptstyle RC}^2-1}{\tau_{\scriptscriptstyle RC}^2+1}\bigg)\bigg) \\
		&\quad\quad\quad + 2 \Big(( p_{10} + q_{01} ) \Big((-1)^{(1 + \mu_1)} \tau_{\scriptscriptstyle RC} + \pi \mu_1 (\tau_{\scriptscriptstyle RC}^2-1)\Big) - v_{01} \omega_{\scriptscriptstyle RC} \log(2)\Big)\bigg),\\
		& p_{10} = - q_{01},\\
		& r_{00} = \frac{1}{4\omega_{\scriptscriptstyle LS}}\bigg(4 r_{10} \omega_{\scriptscriptstyle LC} + 4 b_{\scriptscriptstyle L} (u_{00} + v_{01}) \omega_{\scriptscriptstyle LC} + 2 b_{\scriptscriptstyle L} (r_{10} + s_{01}) (\pi \mu_2 (1 + \tau_{\scriptscriptstyle LC}^2)-\tau_{\scriptscriptstyle LC}) \\
		&\quad\quad\quad + (-1)^{\mu_2} b_{\scriptscriptstyle L} (r_{10} + s_{01}) (1 + \tau_{\scriptscriptstyle LC}^2)\arccos\bigg(\frac{\tau_{\scriptscriptstyle LC}^2-1}{\tau_{\scriptscriptstyle LC}^2+1}\bigg)\bigg),\\
		& r_{10} = - s_{01},
	\end{aligned}
\end{equation*}
we have that 
$$C_{1}^0=C_{1}^1=C_{1}^2=C_{2}^0=C_{2}^1=C_{2}^2=0,\quad\text{with}\quad\text{rank}\frac{\partial(C_{1}^0,C_{1}^1,C_{1}^2,C_{2}^0,C_{2}^1)}{\partial(p_{00},u_{10},p_{10},r_{00},r_{10})}=5.$$
Moreover, $C_{2}^2=\alpha_1^2 C_{1}^2+\alpha_1^0 C_{1}^0+\alpha_1^1 C_{1}^1+\alpha_2^0 C_{2}^0+\alpha_2^1 C_{2}^1$. For this case, with $a_{\scriptscriptstyle R}=0$, $b_{\scriptscriptstyle R}=a_{\scriptscriptstyle L}=1$, $b_{\scriptscriptstyle L}=2$, $c_{\scriptscriptstyle R}=c_{\scriptscriptstyle L}=-1$, $\beta_{\scriptscriptstyle R}\in\{0,2\}$ and $\beta_{\scriptscriptstyle L}\in\{-3,-1/2\}$, the right subsystem from $\eqref{eq:01}|_{\epsilon=0}$ has a virtual (resp. real) center at the point $(0,0)$ (resp. $(2,0)$) when $\beta_{\scriptscriptstyle R}=0$ (resp. $\beta_{\scriptscriptstyle R}=2$) and the left subsystem has a virtual (resp. real) center at the point $(0,-1/2)$ (resp. $(-5,2)$) when  $\beta_{\scriptscriptstyle L}=-1/2$ (resp. $\beta_{\scriptscriptstyle L}=-3$). Moreover, we have that $\alpha_1^2\approx-1.3811$ when $\beta_{\scriptscriptstyle R}=0$ and $\beta_{\scriptscriptstyle L}=-1/2$, $\alpha_1^2\approx-27.0327$ when $\beta_{\scriptscriptstyle R}=0$ and $\beta_{\scriptscriptstyle L}=-3$,    $\alpha_1^2\approx-1.0591$ when $\beta_{\scriptscriptstyle R}=2$ and $\beta_{\scriptscriptstyle L}=-3$. This implies, by Proposition \ref{the:coef}, that $C_{1}^0$, $C_{1}^1$, $C_{1}^2$, $C_{2}^0$ and $C_{2}^1$ can be taken as free coefficients, such that $\max \{N_{M_0}+N_{M_1}+N_{M_2}\}$ is at least six in a neighborhood of $h=1$. More precisely, the configuration of zeros $(2,2,2)$ is achievable. Therefore, the number of limit cycles from system \eqref{eq:01} that can bifurcate of the period annulus near $h=1$, when system $\eqref{eq:01}|_{\epsilon=0}$ is of the type CSC, is at least six.
\end{proof}

\begin{remark}
Through numerical simulations, apparently for the case SSS, we always have $\alpha_1^2>0$. Therefore, we cannot guarantee more than 5 limit cycles using our method proposed in Proposition \ref{the:coef}.
\end{remark}

\section{Appendix A.}
The coefficients list of the Melnikov functions for the case CSS is the following:
\begin{equation*}
	\begin{aligned}
		& C_{0}^0 = 2 (p_{00} + p_{10}) + \frac{2b_{\scriptscriptstyle R}}{b_{\scriptscriptstyle L}}(r_{10}-r_{00})+b_{\scriptscriptstyle R}\bigg(2 v_{01}-2 u_{10}+\frac{1}{\omega_{\scriptscriptstyle LC}}(r_{10} + s_{01}) ( \pi \mu_2 (1 + \tau_{\scriptscriptstyle LC}^2)\\
		&\quad\quad\,\, -\tau_{\scriptscriptstyle LC})+\frac{\tau_{\scriptscriptstyle RS}}{\omega_{\scriptscriptstyle RS}}(p_{10} + q_{01})\bigg)+\frac{(-1)^{\mu_2}b_{\scriptscriptstyle R}}{2\omega_{\scriptscriptstyle LC}}(r_{10} + s_{01})(\tau_{\scriptscriptstyle LC}^2+1)\arccos\bigg(\frac{\tau_{\scriptscriptstyle LC}^2-1}{\tau_{\scriptscriptstyle LC}^2+1}\bigg)\\
		&\quad\quad\,\, -\frac{b_{\scriptscriptstyle R}}{2\omega_{\scriptscriptstyle RS}}(p_{10} + q_{01}) (\tau_{\scriptscriptstyle RS}^2-1)\log\bigg(\frac{\tau_{\scriptscriptstyle RS}+1}{\tau_{\scriptscriptstyle RS}-1}\bigg),
	\end{aligned}
\end{equation*}
\begin{equation*}
	\begin{aligned}
		& C_{0}^1 = 2 (p_{00} + p_{10}) + \frac{(-1)^{\mu_2}b_{\scriptscriptstyle R}}{\omega_{\scriptscriptstyle LC}}(r_{10} + s_{01})\arccos\bigg(\frac{\tau_{\scriptscriptstyle LC}^2-1}{\tau_{\scriptscriptstyle LC}^2+1}\bigg) + b_{\scriptscriptstyle R}\bigg(\frac{2}{b_{\scriptscriptstyle L}}( r_{10}-r_{00})+\frac{1}{\omega_{\scriptscriptstyle LC}}\\
		&\quad\quad\quad (r_{10} + s_{01}) (2 \pi \mu_2 + ( (-1)^{\mu_2}-1 ) \tau_{\scriptscriptstyle LC})+u_{10} (\log(4)-2 ) + v_{01} (2 + \log(4))\bigg)\\
		&\quad\quad\quad+\frac{b_{\scriptscriptstyle R}}{\omega_{\scriptscriptstyle RS}}(p_{10} + q_{01})\log\bigg(\frac{\tau_{\scriptscriptstyle RS}+1}{\tau_{\scriptscriptstyle RS}-1}\bigg),
	\end{aligned}
\end{equation*}
\begin{equation*}
	\begin{aligned}
		& C_{0}^2 = \frac{b_{\scriptscriptstyle R}}{2\omega_{\scriptscriptstyle LC}\omega_{\scriptscriptstyle RS}}\bigg(2 \bigg(\pi (r_{10} + s_{01}) \mu_2 \omega_{\scriptscriptstyle RS} + (u_{10} + v_{01}) \omega_{\scriptscriptstyle LC} \omega_{\scriptscriptstyle RS}+\frac{(-1)^{\mu_2}\tau_{\scriptscriptstyle LC}\omega_{\scriptscriptstyle RS}}{\tau_{\scriptscriptstyle LC}^2+1}(r_{10} + s_{01})\\
		&\quad\quad\quad \frac{\tau_{\scriptscriptstyle RS}\omega_{\scriptscriptstyle LC}}{\tau_{\scriptscriptstyle RS}^2-1}(p_{10} + q_{01})\bigg)+(-1)^{\mu_2} (r_{10} + s_{01}) \omega_{\scriptscriptstyle RS}\arccos\bigg(\frac{\tau_{\scriptscriptstyle LC}^2-1}{\tau_{\scriptscriptstyle LC}^2+1}\bigg)+(u_{10} + v_{01}) \omega_{\scriptscriptstyle LC} \\
		&\quad\quad\quad\omega_{\scriptscriptstyle RS} \log(4) + (p_{10} + q_{01}) \omega_{\scriptscriptstyle LC}\log\bigg(\frac{\tau_{\scriptscriptstyle RS}+1}{\tau_{\scriptscriptstyle RS}-1}\bigg)\bigg),
	\end{aligned}
\end{equation*}
\begin{equation*}
	\begin{aligned}
		& C_{0}^3 = \frac{b_{\scriptscriptstyle R}}{12}\bigg(3 (u_{10} +  v_{01})+\frac{8(-1)^{\mu_2}\tau_{\scriptscriptstyle LC}^3}{\omega_{\scriptscriptstyle LC}(\tau_{\scriptscriptstyle LC}^2+1)^2}(r_{10} + s_{01})+\frac{8\tau_{\scriptscriptstyle RS}^3}{\omega_{\scriptscriptstyle RS}(\tau_{\scriptscriptstyle RS}^2-1)^2}(p_{10} + q_{01}),\bigg)\quad\quad
	\end{aligned}
\end{equation*}
\begin{equation*}
	\begin{aligned}
		& C_{1}^0 = 2( p_{00} + p_{10}) + b_{\scriptscriptstyle R} ( v_{01} - 2 u_{00} - u_{10}) + \frac{b_{\scriptscriptstyle R}\tau_{\scriptscriptstyle RS}}{\omega_{\scriptscriptstyle LC}}(p_{10} + q_{01})-\frac{b_{\scriptscriptstyle R}}{2\omega_{\scriptscriptstyle RS}}(p_{10} + q_{01})\quad\quad\quad\\
		&\quad\quad\quad(\tau_{\scriptscriptstyle RS}^2-1)\log\bigg(\frac{\tau_{\scriptscriptstyle RS}+1}{\tau_{\scriptscriptstyle RS}-1}\bigg),\\
		& C_{1}^1 = 2( p_{00} + p_{10}) + b_{\scriptscriptstyle R} \Big( v_{01} - 2 u_{00} - u_{10}+(u_{10} + v_{01}) \log(2)\Big)+\frac{b_{\scriptscriptstyle R}}{\omega_{\scriptscriptstyle RS}}(p_{10} + q_{01})\\
		&\quad\quad\quad\log\bigg(\frac{\tau_{\scriptscriptstyle RS}+1}{\tau_{\scriptscriptstyle RS}-1}\bigg),
	\end{aligned}
\end{equation*}
\begin{equation*}
	\begin{aligned}
		& C_{1}^2 = \frac{b_{\scriptscriptstyle R}}{2}\bigg(\frac{2\tau_{\scriptscriptstyle RS}}{\omega_{\scriptscriptstyle RS}(\tau_{\scriptscriptstyle RS}^2-1)}(p_{10} + q_{01})+(u_{10} + v_{01}) (1 + \log(2))+\frac{1}{\omega_{\scriptscriptstyle RS}}(p_{10} + q_{01})\quad\quad\quad\quad\\
		&\quad\quad\quad\log\bigg(\frac{\tau_{\scriptscriptstyle RS}+1}{\tau_{\scriptscriptstyle RS}-1}\bigg)\bigg),\\
		& C_{2}^0 = \frac{2}{b_{\scriptscriptstyle L}}( r_{10}-r_{00})+2 u_{00} - u_{10} + v_{01} + \frac{1}{\omega_{\scriptscriptstyle LC}}(r_{10} + s_{01}) ( \pi \mu_2 (1 + \tau_{\scriptscriptstyle LC}^2)-\tau_{\scriptscriptstyle LC} )\\
		&\quad\quad\quad\frac{(-1)^{\mu_2}(\tau_{\scriptscriptstyle LC}^2+1)}{2\omega_{\scriptscriptstyle LC}}(r_{10} + s_{01})\arccos\bigg(\frac{\tau_{\scriptscriptstyle LC}^2-1}{\tau_{\scriptscriptstyle LC}^2+1}\bigg),
	\end{aligned}
\end{equation*}
\begin{equation*}
	\begin{aligned}
		& C_{2}^1 = \frac{2}{b_{\scriptscriptstyle L}}( r_{10}-r_{00})+2 u_{00} - u_{10} + v_{01} + \frac{1}{\omega_{\scriptscriptstyle LC}}(r_{10} + s_{01})(2 \pi \mu_2 + ( (-1)^{\mu_2}-1)\quad\quad\quad\quad\\
		&\quad\quad\quad \tau_{\scriptscriptstyle LC})+(u_{10} + v_{01}) \log(2) + \frac{(-1)^{\mu_2}}{\omega_{\scriptscriptstyle LC}}(r_{10} + s_{01})\arccos\bigg(\frac{\tau_{\scriptscriptstyle LC}^2-1}{\tau_{\scriptscriptstyle LC}^2+1}\bigg),\\
		& C_{2}^2 = \frac{1}{2}\bigg(\frac{2}{\omega_{\scriptscriptstyle LC}}\bigg(\pi \mu_2+\frac{(-1)^{\mu_2}\tau_{\scriptscriptstyle LC}}{\tau_{\scriptscriptstyle LC}^2+1}\bigg)(r_{10} + s_{01})+\frac{(-1)^{\mu_2}}{\omega_{\scriptscriptstyle LC}}(r_{10} + s_{01})\\
		&\quad\quad\quad \arccos\bigg(\frac{\tau_{\scriptscriptstyle LC}^2-1}{\tau_{\scriptscriptstyle LC}^2+1}\bigg)+(u_{10} + v_{01}) (1 + \log(2))\bigg).
	\end{aligned}
\end{equation*}


\section{Appendix B.}
The coefficients list of the Melnikov functions for the case CSC is the following:
\begin{equation*}
	\begin{aligned}
		& C_{0}^0 = \frac{1}{2}\bigg(4 (p_{00} + p_{10}) + 2 b_{\scriptscriptstyle R}\bigg(\frac{2}{b_{\scriptscriptstyle L}}( r_{10}-r_{00})-2 (u_{10} +v_{01})-\frac{1}{\omega_{\scriptscriptstyle LC}}(r_{10} + s_{01}) ( \pi \mu_2 (1 + \tau_{\scriptscriptstyle LC}^2)\\
		&\quad\quad\quad-\tau_{\scriptscriptstyle LC})\bigg)+\frac{2 b_{\scriptscriptstyle R}}{\omega_{\scriptscriptstyle RC}} (p_{10} + 
		q_{01}) ( \pi \mu_1 (1 + \tau_{\scriptscriptstyle RC}^2)-\tau_{\scriptscriptstyle RC} )+\frac{(-1)^{\mu_2}b_{\scriptscriptstyle R}}{\omega_{\scriptscriptstyle LC}}(r_{10} + s_{01})(\tau_{\scriptscriptstyle LC}^2+1)\\
		&\quad\quad\quad\arccos\bigg(\frac{\tau_{\scriptscriptstyle LC}^2-1}{\tau_{\scriptscriptstyle LC}^2+1}\bigg)+\frac{(-1)^{\mu_1}b_{\scriptscriptstyle R}}{\omega_{\scriptscriptstyle RC}}(p_{10} + q_{01})(\tau_{\scriptscriptstyle RC}^2+1)\arccos\bigg(\frac{\tau_{\scriptscriptstyle RC}^2-1}{\tau_{\scriptscriptstyle RC}^2+1}\bigg)\bigg),	
	\end{aligned}
	\end{equation*}
\begin{equation*}
	\begin{aligned}
		& C_{0}^1 = \frac{1}{b_{\scriptscriptstyle L}\omega_{\scriptscriptstyle LC}\omega_{\scriptscriptstyle RC}}\bigg(\omega_{\scriptscriptstyle LC}\omega_{\scriptscriptstyle RC}(2 b_{\scriptscriptstyle L} (p_{00} + p_{10}) + 2 b_{\scriptscriptstyle R} (r_{10}-r_{00}) + 2b_{\scriptscriptstyle L}b_{\scriptscriptstyle R} (v_{01}-u_{10}))+b_{\scriptscriptstyle L} b_{\scriptscriptstyle R} \\
		&\quad\quad\quad \bigg(2 \pi( r_{10} + s_{01}) \mu_2 \omega_{\scriptscriptstyle RC} + (-r_{10} - s_{01} + (-1)^{\mu_2} (r_{10} + s_{01})) \tau_{\scriptscriptstyle LC} \omega_{\scriptscriptstyle RC} + p_{10} \omega_{\scriptscriptstyle LC} (2 \pi \mu_1  \\
		&\quad\quad\quad+(-1 + (-1)^{\mu_1}) \tau_{\scriptscriptstyle LC}) + q_{01} (2 \pi \mu_1 \omega_{\scriptscriptstyle LC} - \omega_{\scriptscriptstyle LC} \tau_{\scriptscriptstyle RC} + (-1)^{\mu_1} \omega_{\scriptscriptstyle LC} \tau_{\scriptscriptstyle RC})\bigg)+b_{\scriptscriptstyle L} b_{\scriptscriptstyle R} \bigg((-1)^{\mu_2} \\
		&\quad\quad\quad(r_{10} + s_{01})\omega_{\scriptscriptstyle RC}\arccos\bigg(\frac{\tau_{\scriptscriptstyle LC}^2-1}{\tau_{\scriptscriptstyle LC}^2+1}\bigg)+(-1)^{\mu_1} (p_{10} + q_{01}) \omega_{\scriptscriptstyle LC}\arccos\bigg(\frac{\tau_{\scriptscriptstyle RC}^2-1}{\tau_{\scriptscriptstyle RC}^2+1}\bigg)\\
		&\quad\quad\quad +(u_{10} + v_{01}) \omega_{\scriptscriptstyle LC} \omega_{\scriptscriptstyle RC} \log(4)\bigg)\bigg),
	\end{aligned}
\end{equation*}
\begin{equation*}
\begin{aligned}
		& C_{0}^2 = \frac{b_{\scriptscriptstyle L}}{2}\bigg(2 u_{10} + 2\bigg(v_{01}+\frac{1}{\omega_{\scriptscriptstyle RC}}\bigg(\frac{1}{\omega_{\scriptscriptstyle LC}}\bigg(\pi (p_{10} + q_{01}) \mu_1 \omega_{\scriptscriptstyle LC} + \pi (r_{10} + s_{01}) \mu_2 \omega_{\scriptscriptstyle RC} +\frac{(-1)^{\mu_2}}{\tau_{\scriptscriptstyle LC}^2+1}\\
		&\quad\quad\quad \tau_{\scriptscriptstyle LC} \omega_{\scriptscriptstyle LC}(r_{10} + s_{01}) \bigg)+\frac{(-1)^{\mu_1}\tau_{\scriptscriptstyle RC}}{\tau_{\scriptscriptstyle RC}^2+1}(p_{10} + q_{01})\bigg)\bigg)+\frac{(-1)^{\mu_2}}{\omega_{\scriptscriptstyle LC}}(r_{10} + s_{01})\\
		&\quad\quad\quad\arccos\bigg(\frac{\tau_{\scriptscriptstyle LC}^2-1}{\tau_{\scriptscriptstyle LC}^2+1}\bigg)+\frac{(-1)^{\mu_1}}{\omega_{\scriptscriptstyle RC}}(p_{10} + q_{01})\arccos\bigg(\frac{\tau_{\scriptscriptstyle RC}^2-1}{\tau_{\scriptscriptstyle RC}^2+1}\bigg)+(u_{10} + v_{01}) \log(4)\bigg),
	\end{aligned}
\end{equation*}
\begin{equation*}
\begin{aligned}
	& C_{0}^3 = \frac{b_{\scriptscriptstyle R}}{12}\bigg(3 (u_{10} +  v_{01})+\frac{8(-1)^{\mu_2}\tau_{\scriptscriptstyle LC}^3}{\omega_{\scriptscriptstyle LC}(\tau_{\scriptscriptstyle LC}^2+1)^2}(r_{10} + s_{01})+\frac{8(-1)^{\mu_1}\tau_{\scriptscriptstyle RC}^3}{\omega_{\scriptscriptstyle RC}(\tau_{\scriptscriptstyle RC}^2-1)^2}(p_{10} + q_{01})\bigg),\quad\quad\quad\\
	& C_{1}^0 = \frac{1}{2\omega_{\scriptscriptstyle RC}}\bigg(4 (p_{00} + p_{10}) \omega_{\scriptscriptstyle RC} - 2 b_{\scriptscriptstyle R} (2 u_{00} + u_{10} - v_{01}) \omega_{\scriptscriptstyle RC} + 2 b_{\scriptscriptstyle R} (p_{10} + q_{01}) ( \pi \mu_1 (1 + \tau_{\scriptscriptstyle RC}^2) \\
	&\quad\quad\quad -\tau_{\scriptscriptstyle RC})+ (-1)^{\mu_1} b_{\scriptscriptstyle R} (p_{10} + q_{01}) (1 + \tau_{\scriptscriptstyle RC}^2)\arccos\bigg(\frac{\tau_{\scriptscriptstyle RC}^2-1}{\tau_{\scriptscriptstyle RC}^2+1}\bigg)\bigg),
\end{aligned}
\end{equation*}
\begin{equation*}
\begin{aligned}
	& C_{1}^1 = \frac{1}{\omega_{\scriptscriptstyle RC}}\bigg(2 (p_{00} + p_{10}) \omega_{\scriptscriptstyle RC} + b_{\scriptscriptstyle R} (p_{10} + 
	q_{01}) (2 \pi \mu_1 + ((-1)^{\mu_1}-1) \tau_{\scriptscriptstyle RC}) + (-1)^{\mu_1} b_{\scriptscriptstyle R}\quad\quad\quad \\
	&\quad\quad\quad (p_{10} + q_{01})\arccos\bigg(\frac{\tau_{\scriptscriptstyle RC}^2-1}{\tau_{\scriptscriptstyle RC}^2+1}\bigg)+b_{\scriptscriptstyle R} \omega_{\scriptscriptstyle RC} ( v_{01}-2 u_{00} - u_{10}  + (u_{10} + v_{01}) \log(2))\bigg),\\
	& C_{1}^2 = \frac{b_{\scriptscriptstyle R}}{2\omega_{\scriptscriptstyle RC}(\tau_{\scriptscriptstyle RC}^2+1)}\bigg((-1)^{\mu_1}  (p_{10} + q_{01})\bigg(2 \tau_{\scriptscriptstyle RC} + (1 + \tau_{\scriptscriptstyle RC}^2)\arccos\bigg(\frac{\tau_{\scriptscriptstyle RC}^2-1}{\tau_{\scriptscriptstyle RC}^2+1}\bigg)\bigg)\\
	&\quad\quad\quad + (1 + \tau_{\scriptscriptstyle RC}^2) \bigg(2 \pi (p_{10} + q_{01}) \mu_1 + (u_{10} + v_{01}) \omega_{\scriptscriptstyle RC} (1 + \log(2))\bigg)\bigg),
\end{aligned}
\end{equation*}
\begin{equation*}
\begin{aligned}
	& C_{2}^0 = \frac{2}{b_{\scriptscriptstyle L}}(r_{10}-r_{00})+2 u_{00} - u_{10} + v_{01} +\frac{1}{\omega_{\scriptscriptstyle LC}}(r_{10} + s_{01}) ( \pi \mu_2 (1 + \tau_{\scriptscriptstyle LC}^2)-\tau_{\scriptscriptstyle LC})+\frac{(-1)^{\mu_2}}{2\omega_{\scriptscriptstyle LC}}\\
	&\quad\quad\quad (r_{10} + s_{01}) (\tau_{\scriptscriptstyle LC}^2+1)\arccos\bigg(\frac{\tau_{\scriptscriptstyle LC}^2-1}{\tau_{\scriptscriptstyle LC}^2+1}\bigg),\\
	&C_{2}^1 = \frac{2}{b_{\scriptscriptstyle L}}(r_{10}-r_{00})+2 u_{00} - u_{10} + v_{01} +\frac{1}{\omega_{\scriptscriptstyle LC}}(r_{10} + s_{01}) (2 \pi \mu_2 + ((-1)^{\mu_2}-1) \tau_{\scriptscriptstyle LC})\\
	&\quad\quad\quad +\frac{(-1)^{\mu_2}}{\omega_{\scriptscriptstyle LC}}(r_{10} + s_{01}) \arccos\bigg(\frac{\tau_{\scriptscriptstyle LC}^2-1}{\tau_{\scriptscriptstyle LC}^2+1}\bigg)+(u_{10} + v_{01}) \log(2),
\end{aligned}
\end{equation*}
\begin{equation*}
\begin{aligned}
	&C_{2}^2 = \frac{1}{2}\bigg(\frac{2}{\omega_{\scriptscriptstyle LC}}(r_{10} + s_{01})\bigg(\pi \mu_2+\frac{(-1)^{\mu_2}\tau_{\scriptscriptstyle LC}}{\tau_{\scriptscriptstyle LC}^2+1}\bigg)+\frac{(-1)^{\mu_2}}{\omega_{\scriptscriptstyle LC}}(r_{10} + s_{01})\arccos\bigg(\frac{\tau_{\scriptscriptstyle LC}^2-1}{\tau_{\scriptscriptstyle LC}^2+1}\bigg)\\
	&\quad\quad\quad+(u_{10} + v_{01}) (1 + \log(2))\bigg).
	\end{aligned}
\end{equation*}


\section{Acknowledgments}
The first author is partially supported by Pronex/FAPEG/CNPq grants 2012 10 26 7000 803 and  2017 10 26 7000 508, CAPES grant 88881.068462/2014-01 and Universal/CNPq grant 420858/2016-4.
The second, third and fourth authors are partially supported by S\~ao Paulo Research Foundation (FAPESP) grants 22/04040-6, 22/09633-5 and 18/19726-5, respectively.
The third and fourth authors are also supported by S\~ao Paulo Research Foundation (FAPESP) grant  19/10269-3.
The fifth author is supported by CAPES grant 88882.434343/2019-01. Moreover, this article was possible thanks to the scholarship granted from the Brazilian Federal Agency for Support and Evaluation of Graduate Education (CAPES), in the scope of the Program CAPES-Print, process number 88887.310463/2018-00, International Cooperation Project number 88881.310741/2018-01.


\end{document}